\documentclass[10pt, a4paper, reqno, english]{amsart}

\usepackage{hyperref}
\usepackage{amssymb,mathrsfs}
\usepackage[all]{xy}
\usepackage{enumitem}
\usepackage{tikz}
\usetikzlibrary{matrix,arrows}
\usepackage{comment}

\usepackage{mathscinet}

\newcommand\ZZ{\mathbb{Z}}

\newcommand\NN{\mathbb{N}}

\newcommand\CC{\mathbb{C}}
\newcommand\RR{\mathbb{R}}

\newcommand\QQ{\mathbb{Q}}

\newcommand{\KK}{\mathbb{K}} 
\newcommand\p{\mathfrak{p}} 
\newcommand{\idealnorm}{\mathfrak{N}}
\newcommand\PP{\mathbb{P}} 
\newcommand{\A}{\mathbb{A}} 
\newcommand{\OO}{\mathcal{O}} 
\newcommand{\ii}[1]{\mathcal{J}_{#1}} 
\newcommand{\ic}[1]{\mathcal{I}_{#1}} 
\newcommand{\gen}[1]{\nu_{#1}} 
\newcommand{\indexset}{{\mathcal I}} 
\newcommand{\ideals}{{\mathcal I_{\KK}}} 
\newcommand{\divclass}{{\mathbf D}} 

\newcommand{\beps}{\boldsymbol{\varepsilon}}
\newcommand{\bgamma}{\boldsymbol{\gamma}}

\DeclareMathOperator{\spec}{Spec} 
\DeclareMathOperator{\pic}{Pic} 
\DeclareMathOperator{\eff}{Eff} 
\DeclareMathOperator{\Hom}{Hom} 
\DeclareMathOperator{\rk}{rk} 
\DeclareMathOperator{\meas}{meas} 

\newtheorem{theorem}{Theorem}
\newtheorem{lemma}[theorem]{Lemma}

\theoremstyle{definition}
\newtheorem{definition}[theorem]{Definition}

\newtheorem{remark}[theorem]{Remark}

\numberwithin{theorem}{section}
\numberwithin{equation}{section}

\begin{document}

\setcounter{tocdepth}{2}

\title{Points of bounded height on certain subvarieties of toric varieties}
  
\author{Marta Pieropan} 
\author{Damaris Schindler} 

\address{Utrecht University, Mathematical Institute, Budapestlaan 6, 3584 CD Utrecht, the Netherlands}
\email{m.pieropan@uu.nl}
\address{Mathematisches Institut, Universitaet Goettingen, Bunsenstrasse 3-5, 37073 Goettingen, Germany}
\email{damaris.schindler@mathematik.uni-goettingen.de}

\date{\today}




\begin{abstract}
We combine the split torsor method and the hyperbola method for toric varieties to count rational points and Campana points of bounded height on certain subvarieties of toric varieties.
\end{abstract}

\maketitle
\tableofcontents

\section{Introduction}
	We combine the split torsor method and the hyperbola method for toric varieties to count rational points
	 and Campana points of bounded height on certain subvarieties of smooth split proper toric varieties. 
	 This line of research has been initiated by Blomer and Br\"udern \cite{BB} in the setting of diagonal hypersurfaces 
	 in products of projective spaces. Other results in this direction include hypersurfaces and complete intersections 
	 in products of projective spaces \cite{Sch}, improvements for bihomogeneous hypersurfaces for degree $(2,2)$ 
	 and $(1,2)$ by Browning and Hu \cite{BroHu} and Hu \cite{Hu}, as well as generalisations to hypersurfaces 
	 in certain toric varieties by Mignot \cite{mignot1, Mignot, mignot2}.\par
	 
	The versions of hyperbola method used in all of these articles are rather close to the original one \cite{BB} 
	for products of projective spaces. In our recent work \cite{hypmethod} we established a very general form 
	of the hyperbola method for split toric varieties, in which the height condition can also globally be given by 
	the maximum of several monomials. The goal of this article is to show applications of our new hyperbola method. 
	We develop a refined framework for the split torsor method on split smooth proper toric varieties and show 
	that counting results for subvarieties of projective spaces can be carried over to toric varieties 
	by a direct application of the hyperbola method \cite{hypmethod}. With this we can prove new cases of 
	Manin's conjecture \cite{MR0974910, MR1032922} on the number of rational points of bounded height 
	on Fano varieties for certain subvarieties in toric varieties.\par
	
	The split torsor method provides a parametrisation of rational points on Fano varieties via 
	Cox rings \cite{MR1679841, MR4177266}. The Cox ring of a smooth proper toric variety $X$ 
	is a polynomial ring endowed with a grading by the Picard group of the toric variety \cite{MR1299003}. 
	Subvarieties of toric varieties are intersections of hypersurfaces, which are defined by $\pic(X)$-homogeneous 
	polynomials in the Cox ring of $X$. The subvarieties considered in this paper are defined by homogeneous elements 
	in the Cox ring of the toric variety such that each polynomial involves only variables of the same degree. 
	With the split torsor method parametrisation, the height is given by the maximum of a set of monomials and 
	this is the correct shape to apply our generalised version of the hyperbola method \cite{hypmethod}. 
	The hyperbola method reduces the counting problem to counting functions over boxes of different shapes. 
	An advantage of our method is that it is already adapted to the shape of height functions appearing. 
	Also, compared to earlier versions of the hyperbola method, we do not need estimates for lower-dimensional 
	boxes, and with this our proofs are relatively short.\par
	
	We now illustrate our approach on a number of examples. In a similar fashion, it is possible to apply 
	counting results such as \cite{birch61, HBdeltamethod, Myerson1, Myerson2, BHBdiffering}, 
	and many others, to subvarieties of toric varieties defined by elements of the Cox ring each involving 
	only variables of the same degree.

\subsection{Results}
Let $X$ be a smooth split complete toric $\QQ$-variety with open torus $T$. Let $\divclass_1,\dots, \divclass_s\in\pic(X)$ be the pairwise distinct classes of the torus-invariant prime divisors on $X$. For $i\in\{1,\dots,s\}$, let $n_i=\dim_\QQ H^0(X,\divclass_i)$. Let $H_L$ be the height associated to a semiample torus-invariant divisor $L$ on $X$ as in Section \ref{sec:heights}. 

Our first result concerns subvarieties of toric varieties defined by linear forms.

\begin{theorem}\label{thm:linear}
Let $V\subseteq X$ be a complete intersection of hypersurfaces $H_{i,l}$ with $1\leq i\leq s$, $1\leq l\leq t_i$. Assume that $[H_{i,l}]=\divclass_i$ in $\pic(X)$ for $i\in\{1,\dots,s\}$ such that $t_i\neq 0$. Assume that $V\cap T\neq\emptyset$ and $t_i\leq n_i-2$ for all $i\in\{1,\dots,s\}$. 
Assume that $L=-(K_X+\sum_{i=1}^s\sum_{l=1}^{t_i}[H_{i,l}])$ is ample.
For $B>0$, let $N_V(B)$ be the number of $\QQ$-rational points on $V\cap T$ of $H_L$-height at most $B$. Then 
$$
N_V(B)=cB(\log B)^{b-1} +O(B(\log B)^{b-2}(\log\log B)^s),
$$
where $b=\rk\pic(V)$, and $c$ is a positive constant, which is defined by \eqref{eq:leading constant} with $k=b-1$, $C_{M,\mathbf d}$ given by \eqref{eq:linear C_M}, and $\varpi_i=n_i-t_i$ for $i\in\{1,\dots,s\}$.
\end{theorem}

We use this result as a toy example to show how to combine the hyperbola method with the universal torsor method in the context of rather general smooth split toric varieties. We now move on to results which require deeper understanding of the underlying Diophantine problems via methods from Fourier analysis.

We start with a result that concerns subvarieties of toric varieties defined by bihomogeneous polynomials. It is obtained combining the framework developed in this paper with the hyperbola method \cite{hypmethod} and preliminary counting results in boxes of different side lengths from \cite{Sch}.

\begin{theorem}\label{thm:bihomog}
Let $V\subseteq X$ be a smooth complete intersection of hypersurfaces $H_1,\dots,H_t$ of the same degree $e_1\divclass_1+e_2\divclass_2$ in $\pic(X)$.  Assume that $V\cap T\neq\emptyset$, that $n_i-te_i\geq 2$ for $i\in\{1,2\}$, and that $n_1+n_2>\dim V_1^*+\dim V_2^* + 3\cdot 2^{e_1+e_2}e_1e_2t^3$, where $V_1^*, V_2^*\subseteq\A^{n_1+n_2}$ are  affine varieties defined in \S\ref{sec:bihomog}.
Assume that $L=-([K_X]+[H_1+\dots+H_t])$ is ample.
Then there is an open subset $W\subseteq X$ such that the number $N_{V,W}(B)$ of $\QQ$-rational points on $V\cap W\cap T$ of $H_L$-height  at most $B$ satisfies 
$$
N_{V,W}(B)=cB(\log B)^{b-1} +O(B(\log B)^{b-2}(\log\log B)^s)
$$
for $B>0$, where $b=\rk\pic(V)$, and $c$ is  defined in \eqref{eq:leading constant} with $k=b-1$,  $C_{M,\mathbf d}$ given by \eqref{eq:bihom C_M},
$\varpi_i=n_i-te_i$ for $i\in\{1,2\}$, and $\varpi_i=n_i$ for $i\in\{3,\dots,s\}$.
The constant $c$ is positive if $V(\QQ_v)\neq\emptyset$ for all places $v$ of $\QQ$.
\end{theorem}

Theorems \ref{thm:linear} and \ref{thm:bihomog} are compatible with Manin's conjecture \cite{MR0974910}, as $L|_V=-K_V$ by adjunction. The proofs in Sections \ref{sec:linear} and \ref{sec:bihomog} yield an asymptotic formula even if we drop the ampleness assumption on $L$.\par
Theorem \ref{thm:bihomog} as well as work of Mignot \cite{Mignot, mignot2} include the case of certain hypersurfaces in products of projective spaces. However, in comparison to Mignot's work we do not require the condition that the effective cone of the toric variety is simplicial. An example of a split toric variety with non-simplicial effective cone, where our theorem applies, is the blow-up at a torus-invariant point of $\mathbb{P}^{n_1}\times \mathbb{P}^{n_2} \times Y$ where $n_1,n_2$ are sufficiently large and $Y$ is a split del Pezzo surface of degree 6.\par

Our last result concerns sets of Campana points in the sense of \cite{PSTVA} for subvarieties defined by diagonal equations. We introduce the following integral models.
Let $\mathscr X$ be the $\ZZ$-toric scheme defined by the fan of $X$. For $i\in\{1,\dots,s\}$, let $\mathscr D_{i,1},\dots, \mathscr D_{i,n_i}$ be the torus-invariant prime divisors on $\mathscr X$ of class $\divclass_i$.

\begin{theorem}\label{thm:Campana}
Let $V\subseteq X$ be an intersection of hypersurfaces $H_1,\dots,H_t$ such that $H_i$ is defined by a homogeneous diagonal polynomial  in the Cox ring of $X$ of degree $e_i\divclass_i$ in $\pic(X)$ and with none of the coefficients equal to zero. 
Let $\mathscr V$ be the closure of $V$ in $\mathscr X$.
For $i\in\{1,\dots,s\}$, fix integers $2\leq m_{i,1}\leq \dots\leq m_{i,n_i}$. Let $\mathscr D_{\mathbf m}=\sum_{i=1}^s\sum_{j=1}^{n_i}(1-\frac 1{m_{i,j}})\mathscr D_{i,j}$.
 Assume that $V\cap T\neq\emptyset$, that $n_1,\dots,n_t\geq 2$, and for $i\in\{1,\dots,s\}$, that $\sum_{j=1}^{n_i}\frac 1{m_{i,j}}>3$, that $\sum_{j=1}^{n_i-1}\frac 1{e_im_{i,j}(e_im_{i,j}+1)}\geq 1$ if $e_i=1$,
 and $\sum_{j=1}^{n_i}\frac 1{2s_0(e_im_{i,j})}>1$ if $e_i\geq 2$, where $s_0(e_im_{i,j})$ is defined in Lemma \ref{lem:diagonal}.
Let $L=-(K_X+\mathscr D_{\mathbf m}|_X+H_1+\dots+H_t)$ be ample.
For $B>0$, let $N_V(B)$ be the number of $\ZZ$-Campana points on $(\mathscr V, \mathscr D_{\mathbf m}|_{\mathscr V})$ that lie in $T$ and have $H_L$-height at most $B$. Then 
$$
N_V(B)=cB(\log B)^{b-1} +O(B(\log B)^{b-2}(\log\log B)^s),
$$
where $b=\rk\pic(V)$, and $c$ is defined in \eqref{eq:leading constant} with $k=b-1$,  $C_{M,\mathbf d}$ given by \eqref{eq:Campana C_M},  and
$\varpi_1,\dots,\varpi_s$ given by \eqref{eq:Campana varpi}.
\end{theorem}

The order of growth in Theorem \ref{thm:Campana} is compatible with the Manin-type conjecture for Campana points \cite{PSTVA}, as $L|_V$ is the log anticanonical divisor of the pair $(V,\mathscr D_{\mathbf m}|_V)$ by adjunction.


We now give a number of examples, where Theorems  \ref{thm:bihomog} and \ref{thm:Campana} can be applied. Due to the range of application of the circle method, we require the Cox ring of the toric variety to have a large number of variables of the same degree. 
This holds for toric varieties with several torus-invariant prime divisors of the same degree and for products of such toric varieties. Here are some concrete examples:
\begin{itemize}
\item The projective space $\PP^n$ has Cox rings with $n+1$ variables of the same degree.
\item The blow-up of  the projective space $\PP^n$ at $l< n+1$ torus-invariant points has Cox rings with $n+1-l$ variables of the same degree. 
\item 
The blow-up of a product of toric varieties each with several torus-invariant prime divisors of the same degree. Indeed, if $X$ and $Y$ are smooth split toric varieties and the Cox ring of $X$ has $n_X$ variables of the same degree $d_X$, and the Cox ring of $Y$ has $n_Y$ variables of the same degree $d_Y$, and $P\in X\times Y$ is a point where $m_X\leq n_X$ variables of degree $d_X$ vanish and $m_Y\leq n_Y$ variables of degree $d_Y$ vanish, then the Cox ring of the blow-up of $X\times Y$ at $P$ has $m_X$ variables of the same degree $d_X-e$ and $m_Y$ variables of the same degree $d_Y-e$, where $e$ is the class of the exceptional divisor.
\end{itemize}


The structure of this article is as follows. In Section \ref{sec:toricsetting} we reformulate the height function and the multiplicative function $\mu$ for M\"obius inversion according to the principle of grouping variables of the same degree. In Section \ref{sec:subvarieties} we combine the new framework with the hyperbola method developed in \cite{hypmethod} to obtain a general counting tool for points of bounded height on subvarieties of toric varieties. Theorems \ref{thm:linear}, \ref{thm:bihomog}, and \ref{thm:Campana} are proven in Sections \ref{sec:linear}, \ref{sec:bihomog}, and \ref{sec:Campana}, respectively.

\subsection*{Acknowledgements}
We thank for their hospitality  the organizers of the workshop ``Rational Points 2023'' at Schney, where we made significant progress on this project. We are grateful to the Lorentz center in Leiden for their hospitality during the workshop ``Enumerative geometry and arithmetic''. We thank the referee for their comments, that improved the exposition of this article.
The first named author is supported by the NWO grants VI.Vidi.213.019 and OCENW.XL21.XL21.011.
For the purpose of open access, a CC BY public copyright license is applied to any Author Accepted Manuscript version arising from this submission.
\section{Toric varieties setting}
\label{sec:toricsetting}
Here we introduce geometric setup and notation for the whole paper.
We refer the reader to \cite[\S8]{MR1679841} for a concise introduction to toric varieties and their toric models over $\mathbb Z$, and to \cite{MR2810322} for an extensive treatment of toric varieties.

Let $\Sigma$ be the fan of a complete smooth split toric variety $X$ over a number field $\KK$. We denote by $\{\divclass_1,\dots,\divclass_s\}\subseteq\pic(X)$ the set of degrees of prime torus-invariant divisors of $X$. For each $i\in\{1,\dots,s\}$ we denote by $D_{i,1},\dots,D_{i,n_i}$ the torus-invariant divisors of degree $\divclass_i$, and by $\rho_{i,1},\dots,\rho_{i,n_i}$ the corresponding rays of $\Sigma$.	
Let $\indexset:=\{(i,j)\in\NN^2:1\leq i\leq s, 1\leq j\leq n_i\}$.
Let $\Sigma_{\max}$ be the set of maximal cones of $\Sigma$.
For each maximal cone $\sigma$ of $\Sigma$, let $\ii{\sigma}:=\{(i,j)\in \indexset:\rho_{i,j}\subseteq\sigma\}$, let $\ic{\sigma}=\indexset\smallsetminus\ii{\sigma}$, 
and let $I_\sigma$ be the set of indices $i\in\{1,\dots,s\}$ such that $\{(i,1),\dots,(i,n_{i})\}\cap\ic{\sigma}\neq\emptyset$.

Let $\mathscr X$ be the toric scheme defined by $\Sigma$ over $\OO_{\KK}$, and for each $(i,j)\in\indexset$, let $\mathscr D_{i,j}$ be the closure of $D_{i,j}$ in $\mathscr X$.
	
Let $R$ be the polynomial ring over $\OO_{\KK}$ with variables $x_{i,j}$ for $(i,j)\in\indexset$ and endowed with the $\pic(X)$-grading induced by assigning degree $\divclass_i$ to the variable $x_{i,j}$ for all $(i,j)\in\indexset$. 
For every torus-invariant divisor $D=\sum_{i=1}^s\sum_{j=1}^{n_i} a_{i,j} D_{i,j}$ on $X$ and every vector $\mathbf x=(x_{i,j})_{(i,j)\in\indexset}\in \CC^{\indexset}$, we write \[\mathbf x^D:=\prod_{i=1}^s\prod_{j=1}^{n_i}x_{i,j}^{a_{i,j}}.\]
	
	By \cite[\S8]{MR1679841}, $\mathscr X$ has a unique universal torsor $\pi:\mathscr Y\to\mathscr X$, and 
	 $\mathscr Y\subseteq\A^{\sharp \indexset}_{\OO_{\KK}}$ is the open subset whose complement is defined by $\mathbf x^{D_\sigma}=0$ for all 
	maximal cones $\sigma$ of $\Sigma$, where 
	$
	D_\sigma:=\sum_{(i,j)\in\ic{\sigma}}D_{i,j}
	$ 
	for all $\sigma\in\Sigma_{\max}$.
	
	Let $r$ be the rank of $\pic(X)$.
	Let $\mathcal{C}$ be a set of ideals of $\OO_{\KK}$ that form a system of representatives 
	for the class group of $\KK$. 
	As in \cite[\S 6.1]{hypmethod}, we fix a basis of $\pic(X)$, and
	for every divisor $D$ on $X$ and every tuple $\mathfrak c=(\mathfrak c_1,\dots,\mathfrak c_r)\in\mathcal C^r$,
	we write $\mathfrak c^{[D]}:=\prod_{i=1}^r \mathfrak c_i^{b_i}$ where 
	$[D]=(b_1,\dots,b_r)$  with respect to the fixed basis of $\pic(X)$.
	Then, as in \cite[\S2]{MR3514738},
 	\[
 	X(\KK)=\mathscr X(\OO_{\KK})=
 	\bigsqcup_{\mathfrak c\in\mathcal C^{r}}\pi^{\mathfrak c}
 	\left(\mathscr{Y}^{\mathfrak c}(\OO_{\KK})\right),
 	\]
 	where  $\pi^{\mathfrak c}:\mathscr{Y}^{\mathfrak c}\to\mathscr X$ is the twist 
 	of $\pi$ defined in \cite[Theorem 2.7]{MR3552013}.
 	The fibers of $\pi|_{\mathscr{Y}^{\mathfrak c}(\OO_{\KK})}$ are all isomorphic to 
 	$(\OO_{\KK}^\times)^{r}$, and
	$\mathscr{Y}^{\mathfrak c}(\OO_{\KK})\subseteq \OO_{\KK}^{\indexset}$ is the subset of 
	points $\mathbf x\in \bigoplus_{(i,j)\in\indexset}\mathfrak c^{[D_{i,j}]}$ that satisfy
	\begin{equation}
	\label{eq:coprimality_condition}
	\sum_{\sigma\in\Sigma_{\max}}\mathbf x^{D_\sigma}\mathfrak c^{-[D_\sigma]}=\OO_{\KK}.
	\end{equation}

	Let $N$ be the lattice of cocharacters of $X$. Then $\Sigma\subseteq N\otimes_\ZZ\RR$.
	For every $(i,j)\in\indexset$, let $\gen{i,j}$ be the unique generator of $\rho_{i,j}\cap N$.
	For every torus-invariant $\QQ$-divisor $D=\sum_{i=1}^s\sum_{j=1}^{n_i} a_{i,j} D_{i,j}$ of $X$ and for every 
	$\sigma\in\Sigma_{\max}$, let $u_{\sigma,D}\in\Hom_{\ZZ}(N,\QQ)$ be the character determined by 
	$u_{\sigma,D}(\gen{i,j})=a_{i,j}$ for all $(i,j)\in\ii{\sigma}$, and define 
	$D(\sigma):=D-\sum_{i=1}^s\sum_{j=1}^{n_i}u_{\sigma,D}(\gen{i,j})D_{i,j}$. 
	Then $D$ and $D(\sigma)$ are linearly equivalent.
	
\subsection{Torus-invariant divisors}

	Here we collect some properties of toric varieties and their torus-invariant divisors.

\begin{lemma}\label{lem:grouping_geom} \ 
\begin{enumerate}[label=(\roman*), ref=(\roman*)]
\item \label{item:grouping_geom_sigma}
 Let $\sigma\in\Sigma_{\max}$. 
\begin{enumerate}
\item For $i\in I_\sigma$,  there is a unique index $j_{i,\sigma}\in\{1,\dots,n_i\}$ such that $(i, j_{i,\sigma})\in\ic{\sigma}$. So $\sharp I_\sigma=\sharp \ic{\sigma}=r$.
\item For $i\in I_\sigma$, $(i,j')\in\ii{\sigma}$ for all $j'\in\{1,\dots,n_i\}\smallsetminus\{j_{i,\sigma}\}$.
\item For $i\in \{1,\dots,s\}\smallsetminus I_\sigma$, $\{(i,1),\dots,(i,n_{i})\}\subseteq\ii{\sigma}$.
\end{enumerate}
\end{enumerate}
Let $D$ be a torus-invariant $\QQ$-divisor on $X$.
For $\sigma\in\Sigma_{\max}$, write $$D(\sigma)=\sum_{i=1}^s\sum_{j=1}^{n_i}\alpha_{i,j,\sigma}D_{i,j}. $$ For $i\in\{1,\dots,s\}$, let $\alpha_{i,\sigma}=\sum_{j=1}^{n_i}\alpha_{i,j,\sigma}$. 
\begin{enumerate}[resume, label=(\roman*), ref=(\roman*)]
\item \label{item:grouping_geom_Lsigma} 
Let $\sigma\in\Sigma_{\max}$. Then $D(\sigma)=\sum_{i\in I_\sigma}\alpha_{i,\sigma}D_{i,j_{i,\sigma}}$.
\item \label{item:grouping_geom_alpha_isigma}
Let $\sigma,\sigma'\in\Sigma_{\max}$.
If there are $i\in I_{\sigma}$ and $j\in\{1,\dots,n_i\}$ such that  $\ii{\sigma}\cap\ii{\sigma'}=\ii{\sigma}\smallsetminus\{(i,j)\}$, then $I_\sigma=I_{\sigma'}$ and $\alpha_{i',\sigma}=\alpha_{i',\sigma'}$ for all $i'\in\{1,\dots,s\}$.
\item \label{item:grouping_geom_special_sigma}
Let $\sigma\in \Sigma_{\max}$ and for every $i\in I_\sigma$, let $j_i\in\{1,\dots,n_i\}$. Then there exists a unique $\sigma'\in\Sigma_{\max}$ such that $I_{\sigma'}=I_\sigma$,  $(i,j_i)\in\ic{\sigma'}$ for $i\in I_\sigma$, and $\alpha_{i,\sigma'}=\alpha_{i,\sigma}$ for  $i\in\{1,\dots,s\}$.
\item \label{item:grouping_equiv_rel}
The relation $\sigma\sim \sigma'$ if and only if $I_\sigma=I_{\sigma'}$ defines an equivalence relation on $\Sigma_{\max}$, and the equivalence class of $\sigma$ has cardinality $\prod_{i\in I_\sigma}n_i$.
\item \label{item:grouping_f}
Let $\mathcal J\subseteq\mathcal I$ minimal for inclusion and such that $ \mathcal J\cap\mathcal I_\sigma\neq\emptyset$ for all $\sigma\in\Sigma_{\max}$. Let $i\in\{1,\dots,s\}$ such that $\{(i,1),\dots,(i,n_i)\}\cap\mathcal J\neq\emptyset$. Then $\{(i,1),\dots,(i,n_i)\}\subseteq\mathcal J$.
\end{enumerate}
\end{lemma}
\begin{proof}
Part \ref{item:grouping_geom_sigma} follows from the fact that $[D_{i,j}]=\divclass_i$ for all $j\in\{1,\dots,n_i\}$ and that the set $\{\divclass_i:i\in I_\sigma\}$ is a basis of $\pic(X)$ by \cite[Theorem 4.2.8]{MR2810322} as $X$ is smooth and proper.

Part \ref{item:grouping_geom_Lsigma} follows from part \ref{item:grouping_geom_sigma} and the fact that by construction $\alpha_{i,j,\sigma}=0$ whenever $(i,j)\in\ii{\sigma}$.

For part \ref{item:grouping_geom_alpha_isigma} we observe that 
if $\sigma\neq\sigma'$, then $\ii{\sigma}=(\ii{\sigma}\cap\ii{\sigma'})\sqcup\{(i,j)\}$ and $\ii{\sigma'}=(\ii{\sigma}\cap\ii{\sigma'})\sqcup\{(i,j_{i,\sigma})\}$, where  $j_{i,\sigma}$ is the index defined in part \ref{item:grouping_geom_sigma}.
Thus $i\in I_\sigma\cap I_{\sigma'}$, and for every index $i'\in\{1,\dots,s\}$ with $i'\neq i$ we have $\ii{\sigma}\cap \{(i',1),\dots,(i',n_{i'})\}=\ii{\sigma'}\cap \{(i',1),\dots,(i',n_{i'})\} \subseteq\ii{\sigma}\cap\ii{\sigma'}$. Recall that $[D(\sigma)]=\sum_{i\in I_\sigma}\alpha_{i,\sigma}\divclass_i$ and $[D(\sigma')]=\sum_{i\in I_{\sigma'}}\alpha_{i,\sigma'}\divclass_i$. Now the result follows as $[D(\sigma)]=[D(\sigma')]$ in $\pic(X)$ and $\{\divclass_i:i\in I_\sigma\}$ is a basis of $\pic(X)$.

For part \ref{item:grouping_geom_special_sigma}, write $I_\sigma=\{i_1,\dots,i_r\}$. 
We construct by induction $\sigma_1,\dots,\sigma_r$ such that for each $l\in\{1,\dots,r\}$, $(i_1,j_{i_1}),\dots,(i_l,j_{i_l})\in\ic{\sigma_l}$, $I_{\sigma_l}=I_\sigma$ and $\alpha_{i,\sigma_l}=\alpha_{i,\sigma}$ for all $i\in\{1,\dots,s\}$. 
If $(i_1,j_{i_1})\in\ic{\sigma}$, let $\sigma_1=\sigma$. 
Otherwise, $(i_1,j_{i_1})\in\ii{\sigma}$ and by \cite[Lemma 8.9]{MR1679841} there is $\sigma_1\in\Sigma_{\max}$ such that $\ii{\sigma_1}\cap\ii{\sigma}=\ii{\sigma}\smallsetminus\{(i_1,j_{i_1})\}$. 
Since $i_1\in I_\sigma$, by part \ref{item:grouping_geom_alpha_isigma} we have $I_{\sigma_1}=I_\sigma$ and $\alpha_{i,\sigma_1}=\alpha_{i,\sigma}$ for all $i\in\{1,\dots,s\}$. 
Assume that we have constructed $\sigma_{l-1}$ for given $l\leq r$. If $(i_l,j_{i_l})\in\ic{\sigma_{l-1}}$, let $\sigma_l=\sigma_{l-1}$. 
Otherwise, $(i_l,j_{i_l})\in\ii{\sigma_{l-1}}$ and by \cite[Lemma 8.9]{MR1679841} there is $\sigma_l\in\Sigma_{\max}$ such that $\ii{\sigma_l}\cap\ii{\sigma_{l-1}}=\ii{\sigma_{l-1}}\smallsetminus\{(i_l,j_{i_l})\}$. 
Since $i_l\in I_{\sigma_{l-1}}$, by part \ref{item:grouping_geom_alpha_isigma} we have $I_{\sigma_l}=I_{\sigma_{l-1}}=I_\sigma$ and $\alpha_{i,\sigma_l}=\alpha_{i,\sigma_{l-1}}=\alpha_{i,\sigma}$ for all $i\in\{1,\dots,s\}$. 
Since $(i_1,j_{i_1}),\dots,(i_{l-1},j_{i_{l-1})}\in\ic{\sigma_{l-1}}$ and $\ii{\sigma_l}=(\ii{\sigma_{l-1}}\cap\ii{\sigma_l})\cup\{(i_l,j_{i_l,\sigma_{l-1}})\}$, where $j_{i_l,\sigma_{l-1}}$ is the index defined in part \ref{item:grouping_geom_sigma}, we conclude that $(i_1,j_{i_1}),\dots,(i_l,j_{i_l})\in\ic{\sigma_l}$. Take $\sigma'=\sigma_r$.
The uniqueness of $\sigma'$ follows from part \ref{item:grouping_geom_sigma}, as $\sigma'$ is completely determined by $\ic{\sigma'}$.

Part \ref{item:grouping_equiv_rel} is a direct consequence of part \ref{item:grouping_geom_special_sigma}.

For part \ref{item:grouping_f},
let $j\in\{1,\dots,n_i\}$ such that $(i,j)\in\mathcal J$. By minimality of $\mathcal J$, there exists $\sigma\in\Sigma_{\max}$ such that $\mathcal J\cap\mathcal I_\sigma=\{(i,j)\}$. If $n_i>1$, let $j'\in\{1,\dots,n_i\}\smallsetminus\{j\}$. By \cite[Lemma 8.9]{MR1679841} there is $\sigma'\in\Sigma_{\max}$ such that $\mathcal J_{\sigma'}\cap\mathcal J_\sigma=\mathcal J_\sigma\smallsetminus\{(i,j')\}$. Hence, $\mathcal I_{\sigma'}=(\mathcal I_\sigma\smallsetminus\{(i,j)\})\cup\{(i,j')\}$. Since $\mathcal J\cap (\mathcal I_\sigma\smallsetminus\{(i,j)\})=\emptyset$ and $\mathcal J\cap\mathcal I_{\sigma'}\neq\emptyset$, we conclude that $(i,j')\in\mathcal J$.
\end{proof}

\subsection{Heights}
\label{sec:heights}

Let $L$ be a semiample torus-invariant $\QQ$-divisor on $X$. Let $H_L$ be the height on $X$ defined by $L$ as in \cite[\S 6.3]{hypmethod}. For $\sigma\in\Sigma_{\max}$, write $L(\sigma)=\sum_{i=1}^s\sum_{j=1}^{n_i}\alpha_{i,j,\sigma}D_{i,j}$ and $\alpha_{i,\sigma}=\sum_{j=1}^{n_i}\alpha_{i,j,\sigma}$ for all $i\in\{1,\dots,s\}$. Let $\Omega_{\KK}$ be the set of places of $\KK$.

\begin{lemma}
\label{lem:height}
For every $\nu\in\Omega_{\KK}$ and every $\mathbf x\in \mathscr Y(\KK)$, we have 
$$
\sup_{\sigma\in\Sigma_{\max}}|x^{L(\sigma)}|_\nu=\sup_{\sigma\in\Sigma_{\max}}\prod_{i=1}^s\sup_{1\leq j\leq n_i}|x_{i,j}|_\nu^{\alpha_{i,\sigma}}.
$$
\end{lemma}
\begin{proof}
Fix $\nu\in\Omega_{\KK}$ and $\mathbf x\in \mathscr Y(\KK)$.
By Lemma \ref{lem:grouping_geom}\ref{item:grouping_geom_Lsigma}, we have
\begin{gather*} 
 \sup_{\sigma\in\Sigma_{\max}}|x^{L(\sigma)}|_\nu=\sup_{\sigma\in\Sigma_{\max}}\prod_{i\in I_\sigma}|x_{i,{j_{i,\sigma}}}|_\nu^{\alpha_{i,\sigma}}. 
\end{gather*}
For every $i\in\{1,\dots,s\}$, let $j_i\in\{1,\dots,n_{i}\}$ such that $|x_{i,j_i}|_\nu=\sup_{1\leq j\leq n_{i}}|x_{i,j}|_\nu$.
Let $\sigma\in\Sigma_{\max}$. By Lemma \ref{lem:grouping_geom}\ref{item:grouping_geom_special_sigma} there is $\sigma'\in\Sigma_{\max}$ such that $I_{\sigma'}=I_\sigma$, $(i,j_i)\in\ic{\sigma'}$ for all $i\in I_\sigma$, and $\alpha_{i,\sigma'}=\alpha_{i,\sigma}$ for all $i\in\{1,\dots,s\}$.
Then 
\begin{equation*}
|x^{L(\sigma')}|_\nu=\prod_{i\in I_{\sigma'}}|x_{i,j_i}|_\nu^{\alpha_{i,\sigma'}}=\prod_{i\in I_\sigma}\sup_{1\leq j\leq n_{i}}|x_{i,j}|_\nu^{\alpha_{i,\sigma}}=\prod_{i=1}^s\sup_{1\leq j\leq n_{i}}|x_{i,j}|_\nu^{\alpha_{i,\sigma}}.
\qedhere
\end{equation*}
\end{proof}

Thus $H_L(\mathbf x)=\prod_{\nu\in\Omega_{\KK}}\sup_{\sigma\in\Sigma_{\max}}\prod_{i=1}^s\sup_{1\leq j\leq n_{i}}|x_{i,j}|_\nu^{\alpha_{i,\sigma}}$ for all $\mathbf x\in\mathscr Y(\KK)$.

\subsection{Coprimality conditions}
We now rewrite the coprimality condition \eqref{eq:coprimality_condition} in terms of the notation introduced in this paper.

\begin{lemma}\label{lem:grouping_coprimality}
For all $\mathbf x\in \bigoplus_{(i,j)\in\indexset}\mathfrak c^{[D_{i,j}]}$,
$$\sum_{\sigma\in \Sigma_{\max}}\mathbf x^{D_\sigma}\mathfrak c^{-[D_\sigma]}=\sum_{\sigma\in\Sigma_{\max}}\prod_{i\in I_\sigma}(x_{i,1},\dots,x_{i,n_i})\mathfrak c^{-\divclass_i}.$$
\end{lemma}
\begin{proof}
For $\sigma\in\Sigma_{\max}$, let $X_{\sigma}=\{\prod_{i\in I_\sigma}x_{i,j_i}: j_i\in\{1,\dots,n_i\} \ \forall i\in\{1,\dots,s\}\}$.
The inclusion $\subseteq$ is clear as $\mathbf x^{D_\sigma}\in X_{\sigma}$ and $\mathfrak c^{-[D_\sigma]}=\prod_{i\in I_\sigma}\mathfrak c^{-\divclass_i}$ for all $\sigma\in \Sigma_{\max}$. 
For the converse inclusion, fix $\sigma\in\Sigma_{\max}$ and $x\in X_{\sigma}$. For every $i\in I_\sigma$, let $j_i\in\{1,\dots,n_i\}$ such that $x=\prod_{i\in I_\sigma}x_{i,j_i}$.   By Lemma \ref{lem:grouping_geom}\ref{item:grouping_geom_special_sigma} there is $\sigma'\in\Sigma_{\max}$ such that $I_{\sigma'}=I_\sigma$ and $(i,j_{i})\in\ic{\sigma'}$ for $i\in I_\sigma$ . Then $\mathbf x^{D_{\sigma'}}=x$.
\end{proof}

\subsection{M\"obius function}
\label{subsec:mobius}
Let $\ideals$ be the set of nonzero ideals of $\mathcal O_\KK$. Let $\chi:\ideals^s\to\{0,1\}$ be the characteristic function of the subset
\begin{equation}
\label{eq:gcd}
\left\{\mathfrak b\in\ideals^s: \sum_{\sigma\in\Sigma_{\max}}\prod_{i\in I_\sigma}\mathfrak b_i=\mathcal O_{\KK}\right\}.
\end{equation}
For every $\mathfrak d\in\ideals^s$, let $\chi_{\mathfrak d}:\ideals^s\to\{0,1\}$ be the characteristic function of the subset
$$
\left\{ \mathfrak b\in \ideals^s:\mathfrak b_i\subseteq\mathfrak d_i \ \forall i\in\{1,\dots,s\}\right\}.
$$
As in \cite[Lemme 8.5.1]{MR1340296} there exists a unique multiplicative function $\mu:\ideals^s\to\ZZ$ such that 
$$
\chi=\sum_{\mathfrak d\in \ideals^s}\mu(\mathfrak d)\chi_{\mathfrak d}.
$$

Note that if $X=\PP^n_\QQ$, the function $\mu$ defined above coincides with the classical M\"obius function.

\begin{remark} \label{rem:mu}
Let $\mathfrak p\in\mathcal I_{\KK}$ be a prime ideal.
The function $\mu$ is defined recursively by the formula $\mu(\mathfrak b)=\chi(\mathfrak b)-\sum_{\mathfrak b\subsetneq\mathfrak d}\mu(\mathfrak d)$ for every $\mathfrak b\in\ideals^s$, and satisfies the following properties. 
\begin{enumerate}[label=(\roman*), ref=(\roman*)]
\item $\mu(\mathbf 1)=\chi(\mathbf 1)=1$. 
\label{item:mu 1}
\item If $e_i\geq 2$ for some $i\in\{1,\dots,s\}$, then $\mu(\mathfrak p^{e_1},\dots,\mathfrak p^{e_s})=0$, as in that case $\chi(\mathfrak p^{e_1},\dots,\mathfrak p^{e_s})=\chi(\mathfrak p^{e'_1},\dots,\mathfrak p^{e'_s})$ for $e'_i=e_i-1$ and $e'_l=e_l$ for all $l\neq i$. 
\label{item:mu geq2}
\item By induction one shows that $\mu(\mathfrak p^{e_1},\dots,\mathfrak p^{e_s})=0$ whenever $(e_1,\dots,e_s)\neq\mathbf 0$ and there is $\sigma\in\Sigma_{\max}$ such that $e_i=0$ for all $i\in I_\sigma$, as $\chi(\mathfrak p^{e_1},\dots,\mathfrak p^{e_s})=1$ if and only if there is $\sigma\in\Sigma_{\max}$ such that $e_i=0$ for all $i\in I_\sigma$. 
\label{item:mu I}
\item \label{item: mu tildef}
Let 
$$\tilde f:= \min\left \{\sharp J: J\subseteq\{1,\dots, s\}, J\cap I_\sigma\neq\emptyset \ \forall \sigma\in\Sigma_{\max} \right \}.$$ By property \ref{item:mu I}, if $\mu(\mathfrak p^{e_1},\dots,\mathfrak p^{e_s})\neq0$, then there are at least $\tilde f$ indices $i$ with $e_i=1$.
Let $J\subseteq\{1,\dots,s\}$ be smallest with respect to inclusion and such that $J\cap I_\sigma\neq\emptyset$ for all $\sigma\in\Sigma_{\max}$. Let $J'=J\smallsetminus\{j\}$ for some $j\in J$. 
Let $e_i=1$ for $i\in J$ and $e_i=0$ for $i\notin J$. Let $e_i'=e_i$ for $i\neq j$ and $e_j'=0$. Then $\chi(\mathfrak p^{e_1},\dots,\mathfrak p^{e_s})=0$ and $\chi(\mathfrak p^{e'_1},\dots,\mathfrak p^{e'_s})=1$ by minimality of $J$. Thus $\mu(\mathfrak p^{e_1},\dots,\mathfrak p^{e_s})=-1\neq0$. Hence, 
\begin{equation}
\tilde f=\min \left \{\sum_{i=1}^se_i: (e_1,\dots,e_s)\neq\mathbf 0, \mu(\mathfrak p^{e_1},\dots,\mathfrak p^{e_s})\neq0\right \}.
\label{eq:mu tildef}
\end{equation}
\end{enumerate}
\end{remark}

For $\beta=(\beta_1,\dots,\beta_s)\in\RR_{\geq 0}^s$, let 
$$
f_{\beta}:=\min \left \{\sum_{i=1}^s \beta_i e_i: (e_1,\dots,e_s)\neq\mathbf 0, \mu(\mathfrak p^{e_1},\dots,\mathfrak p^{e_s})\neq0\right \}.
$$

\begin{lemma}\label{lem:mobius}
\begin{enumerate}[label=(\roman*), ref=(\roman*)]
\item
\label{item:mobius convergence}
The series 
$$
\sum_{\mathfrak d\in \mathcal I_{\KK}^s}\frac{\mu(\mathfrak d)}{\prod_{i=1}^s\idealnorm(\mathfrak d_i)^{\beta_i}}
$$
converges absolutely if $f_\beta>1$.
\item 
\label{item:mobius positive}
If $f_\beta>1$ and $\beta_1,\dots,\beta_s\in\ZZ_{>0}$, then $\sum_{\mathfrak d\in \mathcal I_{\KK}^s}\frac{\mu(\mathfrak d)}{\prod_{i=1}^s\idealnorm(\mathfrak d_i)^{\beta_i}}>0$.
\end{enumerate}
\end{lemma}
\begin{proof}
For part \ref{item:mobius convergence} we follow the proof of  \cite[Lemma 11.15]{MR1679841} and \cite[Proposition 4]{MR3514738}. 
For $\mathfrak p\in\mathcal I_{\KK}$  prime ideal,
let $S(\mathfrak p)=\sum_{(e_1,\dots,e_s)\in\ZZ_{\geq 0}^s}\frac{|\mu(\mathfrak p^{e_1},\dots,\mathfrak p^{e_s})|}{\prod_{i=1}^s\idealnorm(\mathfrak p)^{\beta_ie_i}}$.
As in \cite[Lemma 11.15]{MR1679841} and \cite[Proposition 4]{MR3514738},
$$
\lim_{b\to\infty}\sum_{\substack{\mathfrak d \in\mathcal I_{\KK}^s\\ \prod_{i=1}^s\idealnorm(\mathfrak d_i)\leq b}}\frac{|\mu(\mathfrak d)|}{\prod_{i=1}^s\idealnorm(\mathfrak d_i)^{\beta_i}}
=\prod_{\mathfrak p}S(\mathfrak p).
$$ 
By Remark \ref{rem:mu}\ref{item:mu geq2} the sum $S(\mathfrak p)$ is finite. By definition of $f_\beta$, if $\mu(\mathfrak p^{e_1},\dots,\mathfrak p^{e_s})\neq 0$ and $(e_1,\dots,e_s)\neq\mathbf 0$, then ${f_\beta}\leq \sum_{i=1}^s\beta_ie_i$. Thus
\begin{gather*}
S(\mathfrak p)=1+\frac 1{\idealnorm(\mathfrak p)^{f_\beta}}Q\left(\frac 1{\idealnorm(\mathfrak p)}\right),
\end{gather*}
where $Q:\RR_{\geq 0}\to \RR_{\geq 0}$ is a monotone increasing function.
Since $\mu(\mathfrak p^{e_1},\dots,\mathfrak p^{e_s})$ is independent of the choice of $\mathfrak p$, the function $Q$ is independent of the choice of $\mathfrak p$. Thus
\begin{gather*}
\sum_{\mathfrak p}\frac 1{\idealnorm(\mathfrak p)^{f_\beta}}Q\left(\frac 1{\idealnorm(\mathfrak p)}\right)
\leq [\KK:\QQ] Q(1) \sum_{n\in\ZZ_{>0}}\frac 1{n^{f_\beta}}. 
\end{gather*}

In part \ref{item:mobius positive} 
 the series is absolutely convergent by part \ref{item:mobius convergence}, hence it suffices to show that each factor of its Euler product $\prod_{\mathfrak p} S_{\mathfrak p}$ is positive. 
For a prime ideal $\mathfrak p\in\mathcal I_{\KK}$, let $\OO_{\mathfrak p}$ be the ring of integers of the completion $\KK_{\mathfrak p}$ of $\KK$ at the valuation $v_{\mathfrak p}$ defined by $\mathfrak p$. Endow $\KK_{\mathfrak p}$ with the Haar measure normalized such that $\OO_{\mathfrak p}$ has volume $1$. Then $\int_{\mathfrak p^j \OO_{\mathfrak p}}dy=\idealnorm(\mathfrak p)^{-j}$ for all $j\geq 0$ by \cite[\S 1.1.13]{MR3838446} and \cite[Proposition II.4.3]{MR1697859}. We denote by $\chi$ the characteristic function of \eqref{eq:gcd} where ideals of $\OO_{\KK}$ are replaced by ideals of $\OO_{\mathfrak p}$.
By Remark \ref{rem:mu}\ref{item:mu geq2},
\begin{gather*}
S_{\mathfrak p}=\sum_{\boldsymbol e\in\{0,1\}^s}\mu(\mathfrak p^{e_1},\dots,\mathfrak p^{e_s})\prod_{i=1}^s\idealnorm(\mathfrak p)^{-e_i\beta_i} 
\\ = \sum_{\boldsymbol e\in\{0,1\}^s}\mu(\mathfrak p^{e_1},\dots,\mathfrak p^{e_s})\prod_{i=1}^s \prod_{j=1}^{\beta_i} \int_{\mathfrak p^{e_i}} \mathrm d y_{i,j}
\\ = \int_{\OO_{\mathfrak p}^{\sum_{i=1}^s \beta_i}}\chi((y_{1,1},\dots,y_{1,\beta_1}),\dots, (y_{s,1},\dots,y_{s,\beta_s}))\prod_{i=1}^s\prod_{j=1}^{\beta_i} \mathrm d y_{i,j}
\\ \geq \int_{(\OO_{\mathfrak p}^\times)^{\sum_{i=1}^s \beta_i}}\prod_{i=1}^s\prod_{j=1}^{\beta_i} \mathrm d y_{i,j} = \left(1-\frac 1{\idealnorm(\mathfrak p)}\right)^{\sum_{i=1}^s \beta_i}>0,
\end{gather*}
as $\chi$ is a nonnegative function with $\chi(\OO_{\mathfrak p},\dots,\OO_{\mathfrak p})=1$.
\end{proof}	

\begin{definition}
A function $A:\ZZ^s_{>0}\to\RR$ is compatible with M\"obius inversion on $X$ if there are $\beta_1,\dots,\beta_s\in\RR^s$ such that $A(\mathbf d)\ll\prod_{i=1}^sd_i^{-\beta_i}$ with $f_{(\beta_1,\dots,\beta_s)}>1$. 
\end{definition}
	
\begin{remark}
\begin{enumerate}[label=(\roman*), ref=(\roman*)]
\item
The inequality $f_\beta>1$ holds whenever $\beta_1,\dots,\beta_s>1$.
\item
If $\beta_1=\dots=\beta_s=1$, then $f_{\beta}=\tilde f$ by \eqref{eq:mu tildef}.
\item (Case $\beta_1=n_1,\dots,\beta_s=n_s$)
 As in \cite[Lemma 11.15(d)]{MR1679841}, let $f$ be the smallest positive integer such that there are $f$ rays of the fan $\Sigma$ that are not contained in a maximal cone. Then $f\geq 2$, as $X$ is proper. Moreover, $$f=\min \left \{\sharp\mathcal J:\mathcal J\subseteq\mathcal I, \mathcal J\cap\mathcal I_\sigma\neq\emptyset \ \forall\sigma\in\Sigma_{\max}\right \},$$
and Remark \ref{rem:mu} combined with Lemma \ref{lem:grouping_geom}\ref{item:grouping_f} gives
\begin{align*}
f&=\min\left \{\sum_{i\in J}n_i: J\subseteq\{1,\dots, s\}, J\cap I_\sigma\neq\emptyset \ \forall \sigma\in\Sigma_{\max}, \sharp J=\tilde f\right \}\\
 &=\min \left \{\sum_{i=1}^sn_ie_i: (e_1,\dots,e_s)\neq\mathbf 0, \mu(\mathfrak p^{e_1},\dots,\mathfrak p^{e_s})\neq0\right \}.
\end{align*}
\end{enumerate}
\end{remark}	

\section{Subvarieties}
\label{sec:subvarieties}

Here we want to count rational points or Campana points of bounded height in subvarieties of toric varieties.

From now on $\KK=\QQ$. Let $X$ be a complete smooth split toric variety as in Section \ref{sec:toricsetting}. Assume that $\rk\pic(X)\geq 2$, that is $X$ is not a projective space. Let $L$ be a semiample toric invariant $\QQ$-divisor on $X$ that satisfies \cite[Assumption 6.3]{hypmethod}. The latter holds, for example, if $L$ is ample. 

Let $g_1,\dots,g_t\in R$ be $\pic(X)$-homogeneous elements. 
Let $V\subseteq X$ be the schematic intersection of the $t$ hypersurfaces defined by $g_1,\dots,g_t$.
Let $T\subseteq X$ be the torus. Without loss of generality, we can assume that $V\cap T\neq\emptyset$. Otherwise, $V$ is contained in a complete smooth split toric subvariety $X'$ of $X$, and we can replace $X$ by $X'$.
Fix $m_{i,j}\in \ZZ_{\geq 1}$ for each $(i,j)\in\indexset$. 
Let $\mathbf m=(m_{i,j})_{(i,j)\in\indexset}$, and $\mathscr D_{\mathbf m}=\sum_{i=1}^s\sum_{j=1}^{n_i}\left(1-\frac 1{m_{i,j}}\right)\mathscr D_{i,j}$. 
Let $\mathscr V$ be the Zariski closure of $V$ in $\mathscr X$. 
We define the intersection multiplicity $n_v(\mathscr D_i|_{\mathscr V},\mathbf x)$ of a point $\mathbf x:\spec \OO_{\KK}\to\mathscr V$ with $\mathscr D_i|_{\mathscr V}$ at a place $v$ of $\KK$ is defined as the colength of the ideal of the fiber product of $\spec\OO_{\KK}\times_{\mathscr V}\mathscr D_i|_{\mathscr V}$ after base change to the completion of $\OO_{\KK}$ at $v$.
This definition coincides with the one in \cite[\S3]{PSTVA} whenever $\mathscr V$ is regular.
Let $\left(\mathscr V,\mathscr D_{\mathbf m}|_{\mathscr V}\right)(\ZZ)$ be the set of Campana $\ZZ$-points on the Campana orbifold  $\left(\mathscr V,\mathscr D_{\mathbf m}|_{\mathscr V}\right)$ as in \cite[Definition 3.4]{PSTVA}.

Let $N_V(B)$ be the number of points in $\left(\mathscr V,\mathscr D_{\mathbf m}|_{\mathscr V}\right)(\ZZ)\cap T(\QQ)$ of height $H_L$ at most $B$. If $m_{i,j}=1$ for all $(i,j)\in\indexset$, then $N_V(B)$ is the set of $\QQ$-rational points on $V\cap T$ of height $H_L$ at most $B$.

For $i\in\{1,\dots,s\}$ and $\mathbf x\in \mathscr Y(\ZZ)$, let $y_i=\sup_{1\leq j\leq n_{i}}|x_{i,j}|$. For $\sigma\in\Sigma_{\max}$, write $L(\sigma)=\sum_{i=1}^s\sum_{j=1}^{n_i}\alpha_{i,j,\sigma}D_{i,j}$ and $\alpha_{i,\sigma}=\sum_{j=1}^{n_i}\alpha_{i,j,\sigma}$ for all $i\in\{1,\dots,s\}$.
Then by \cite[Proposition 6.10]{hypmethod} and Lemma \ref{lem:height},
$$H_L(\mathbf x)=\sup_{\sigma\in \Sigma_{\max}}\prod_{i=1}^s y_i^{\alpha_{i,\sigma}}.$$
By construction, $\left(\mathscr V,\mathscr D_{\mathbf m}|_{\mathscr V}\right)(\ZZ)=\left(\mathscr X,\mathscr D_{\mathbf m}\right)(\ZZ)\cap V(\QQ)$. We use the torsor parameterization of $\left(\mathscr X,\mathscr D_{\mathbf m}\right)(\ZZ)$ from \cite[\S6.4]{hypmethod}.
For $B>0$ and $\mathbf d\in(\ZZ_{>0})^{s}$, 
	let $A(B,\mathbf d)$ be the set of points
	$\mathbf x=(x_{i,j})_{1\leq i\leq s,1\leq j\leq n_i}\in(\ZZ_{\neq0})^{\indexset}$ such that 
	\begin{gather}
	H(\mathbf x)\leq B, \label{cond:height}\\
	d_i\mid x_{i,j} \ \forall i\in\{1,\dots,s\}, \forall j\in\{1,\dots,n_i\}, \label{cond:divisibility}\\
	x_{i,j} \text{ is } m_{i,j}\text{-full} \ \forall i\in\{1,\dots,s\}, \forall j\in\{1,\dots,n_i\}, \label{cond:m-full}\\ 
	g_1=\dots=g_t=0.
	\label{cond:equations}
	\end{gather}	
	We observe that $A(B, \mathbf d)$ is a finite set by \cite[Lemma 6.11]{hypmethod}.	
	Then
	\begin{equation}
	\label{eq:moebius_inversion}
 	N_V(B)=\frac 1{2^r}\sum_{\mathbf d\in(\ZZ_{>0})^{s}}\mu(d)\sharp A(B, \mathbf d)
	\end{equation}
	by Lemma \ref{lem:grouping_coprimality} and the definition of $\mu$ in Section \ref{subsec:mobius}.

	Write
	$$
	\sharp A(B,\mathbf d)=\sum_{\substack{y_1,\dots,y_s\in\ZZ_{>0}
	\\ \prod_{i=1}^s y_i^{\alpha_{i,\sigma}}\leq B, \ \forall \sigma\in\Sigma_{\max}}}f_{\mathbf d}(y_1,\dots,y_s),
	$$
	where 
	\begin{align*}
	f_{\mathbf d}(y_1,\dots,y_s)=\sharp\{\mathbf x\in(\ZZ_{\neq 0})^{\indexset}: 
	\eqref{cond:divisibility},\eqref{cond:m-full},\eqref{cond:equations},
	y_i=\sup_{1\leq j\leq n_{i}}|x_{i,j}| \ \forall i\in\{1,\dots,s\} \}.
	\end{align*}

	 Let $$F_{\mathbf d}(B_1,\dots,B_s)=\sum_{1\leq y_i\leq B_i,1\leq i\leq s}f_{\mathbf d}(y_1,\dots,y_s).$$

\begin{lemma}\label{lem:counting}
Assume that 
\begin{equation}
F_{\mathbf d}(B_1,\dots,B_s)=C_{M,\mathbf d}\prod_{i=1}^sB_i^{\varpi_i} + O\left(C_{E,\mathbf d} \left(\min_{1\leq i\leq s}B_i\right)^{-\epsilon}\prod_{i=1}^sB_i^{\varpi_i}\right)
\end{equation}
with $C_{M,\mathbf d}, C_{E,\mathbf d}, \varpi_1,\dots,\varpi_s,\epsilon>0$ such that $C_{M,\mathbf d}$, $C_{E,\mathbf d}$ are compatible with M\"obius inversion on $X$ as functions of the variables $\mathbf d$.

Let $a$ be the maximal value of $\sum_{i=1}^s\varpi_iu_i$ on the polytope $\mathcal P\subseteq\RR^s$ defined by
$$
\sum_{i=1}^s\alpha_{i,\sigma}u_i\leq 1 \ \forall \sigma\in\Sigma_{\max}, \qquad u_i\geq 0 \ \forall i\in\{1,\dots,s\}.
$$
Let $F$ be the face of $\mathcal P$ where $\sum_{i=1}^s\varpi_iu_i=a$. Let $k$ be the dimension of $F$. 

\begin{enumerate}[label=(\roman*), ref=(\roman*)]
\item 
If $F$  is not contained in a coordinate hyperplane of $\RR^s$,
then 
$$N_V(B)=cB^a(\log B)^k +O(B^a(\log B)^{k-1}(\log\log B)^s),$$
where 
$k$ is the dimension of $F$, and
\begin{equation}
\label{eq:leading constant}
c=(s-1-k)! c_{\mathcal P}2^{-r}\sum_{\mathbf d\in\ZZ_{>0}^s}\mu(\mathbf d)C_{M,\mathbf d}. 
\end{equation}
Here, $c_{\mathcal P}=\lim_{\delta\to 0}\delta^{k+1-s}\meas_{s-1}(H_\delta\cap\mathcal P)$, where $H_\delta\subseteq\RR^s$ is the hyperplane defined by $\sum_{i=1}^s\varpi_iu_i=a-\delta$, and $\meas_{s-1}$ is the $(s-1)$-dimensional measure on $H_\delta$ given by $\prod_{1\leq i\leq s, i\neq\tilde i}(\varpi_i\mathrm du_i)$ for any choice of $\tilde i\in\{1,\dots,s\}$.

\item If $L$ is ample, then
$$
a=\inf\left\{t\in\RR: t[L]-\left [\sum_{i=1}^s\varpi_i\divclass_i\right ] \text{ is effective}\right\},
$$
and $k+1$ is the codimension of the minimal face of the effective cone of $X$ containing $a[L]-[\sum_{i=1}^s\varpi_i\divclass_i]$.

\item If $[L]=\sum_{i=1}^s\varpi_i\divclass_i$ is ample, then the face $F$ is not contained in a coordinate hyperplane, $a=1$ and $k=\rk\pic(X)-1$.

\end{enumerate}
\end{lemma}

\begin{proof}
\begin{enumerate}[label=(\roman*), ref=(\roman*)]
\item
Let $t_i=\varpi_iu_i$ for all $i\in\{1,\dots,s\}$. By the assumptions on $L$, the polytope $\mathcal P$ is bounded and nondegenerate by \cite[Remark 6.2]{hypmethod}. 
Applying \cite[Theorem 1.1]{hypmethod} to $\sharp A(B,\mathbf d)$ gives
$$
N_V(B)=c B^a(\log B)^k + O\left(B^a(\log B)^{k-1}(\log\log B)^s \sum_{\mathbf d\in(\ZZ_{>0})^{s}}\mu(d) C_{E,\mathbf d} \right).
$$
The sums $\sum_{\mathbf d\in(\ZZ_{>0})^{s}}$ in the leading constant $c$ and in the error term converge absolutely by Lemma \ref{lem:mobius} as $C_{M,\mathbf d}$, $C_{E,\mathbf d}$ are compatible with M\"obius inversion on $X$.

\item Let 
	 $$
	 \RR^r\hookrightarrow\RR^s\hookrightarrow\RR^{\indexset}
	 $$
	 be the sequence of injective linear maps dual to 
	 $$d:\bigoplus_{(i,j)\in\indexset}D_{i,j}\ZZ\twoheadrightarrow\bigoplus_{i=1}^s\divclass_i\ZZ\twoheadrightarrow\pic(X).	$$
	 Here, $$\RR^s\hookrightarrow\RR^{\indexset}, \qquad \sum_{i=1}^s u_ie_i\mapsto\sum_{i=1}^s\sum_{j=1}^{n_i}u_ie_{i,j},$$ where $\{e_1,\dots,e_s\}$ denotes the dual basis to $\{\divclass_1,\dots,\divclass_s\}$, and $\{e_{i,j}:(i,j)\in\indexset\}$ denotes the dual basis to $\{D_{i,j}:(i,j)\in\indexset\}$.
	 Let $\widetilde P$ be the polytope defined by 
	 $$
	  \sum_{(i,j)\in\indexset}\alpha_{i,j,\sigma}u_{i,j}\leq 1 \quad \forall \sigma\in\Sigma_{\max}, \qquad u_{i,j}\geq 0 \quad \forall (i,j)\in\indexset.
	 $$
	 Then $\widetilde P\cap \RR^s=P$ and $\left.\sum_{(i,j)\in\indexset}\frac{\varpi_i}{n_i}u_{i,j}\right|_{P}=\sum_{i=1}^s\left(\sum_{j=1}^{n_i}\frac{\varpi_i}{n_i}\right)u_i$. 
	 By \cite[Lemma 6.7]{hypmethod}, the face $F$ of $\widetilde P$ where the maximal value of the function
	  \begin{equation}
	 \sum_{(i,j)\in\indexset}\frac{\varpi_i}{n_i}u_{i,j} \label{eq:maximizing function}
	 \end{equation}
	 is attained is contained in $\widetilde P\cap \RR^r$, and hence also in $P$.
	 Then $a$ is the maximal value of the function \eqref{eq:maximizing function} on $P$. The dual linear programming problem is given by minimizing $\sum_{\sigma\in\Sigma_{\max}}\lambda_\sigma$ on the polytope given by
	 $$
	 \sum_{\sigma\in\Sigma_{\max}}\alpha_{i,j,\sigma}\lambda_\sigma\geq \frac{\varpi_i}{n_i} \quad \forall (i,j)\in\indexset, \qquad \lambda_\sigma\geq 0,\ \forall \sigma\in\Sigma_{\max}.
	 $$
	 The arguments in \cite[\S 6.5.1]{hypmethod} show that
	  $a$ is the smallest real number such that $a[L]-\sum_{i=1}^s\sum_{j=1}^{n_i}\frac{\varpi_i}{n_i}\divclass_i$ is effective. 
	  As in \cite[Proposition 6.13]{hypmethod}, the smallest face of $\eff(X)$ that contains $a[L]-\sum_{i=1}^s\varpi_i\divclass_i$ is dual to the cone generated by $F$ in $\RR^r$, and the latter is defined by $a\sum_{i=1}^s\alpha_{i,\sigma}u_i-\sum_{i=1}^s\varpi_iu_i=0$ for any $\sigma\in\Sigma_{\max}$ such that $F\subseteq \{\sum_{i=1}^s\alpha_{i,j,\sigma}u_{i,j}=1\}$.
	  Thus the minimal face of $\eff(X)$ containing $a[L]-[\sum_{i=1}^s\varpi_i\divclass_i]$ has codimension $k+1$.

\item We argue as in the proof of \cite[Lemma 6.7(ii)]{hypmethod}. Let $\widetilde H\subseteq \RR^s$ be the inclusion dual to the surjection $\bigoplus_{i=1}^s\RR\divclass_i\to\pic(X)\otimes_{\ZZ}\RR$. Then $\sum_{i=1}^s\alpha_{i,\sigma}u_i=\sum_{i=1}^s\varpi_iu_i$ for all $\mathbf u\in \widetilde H$ and all $\sigma\in\Sigma_{\max}$. Thus $\mathcal P\cap\widetilde H$ is the set of elements $\mathbf u$ of $\widetilde H$ such that $u_1,\dots,u_s\geq 0$ and $\sum_{i=1}^s\varpi_iu_i\leq 1$. Since $F\subseteq\widetilde H$ by \cite[Lemma 6.7(i)]{hypmethod}, we have $F=\widetilde H\cap\{\sum_{i=1}^s\varpi_iu_i= 1\}$. As in in the proof of \cite[Lemma 6.7(ii)]{hypmethod} we conclude that $F$ is not contained in a coordinate hyperplane of $\RR^s$.
\end{enumerate}
\end{proof}

\section{Rational points on linear complete intersections}	 
\label{sec:linear}
\begin{proof}[Proof of Theorem \ref{thm:linear}]
For $1\leq i\leq s$ and $1\leq l\leq t_i$, let $g_{i,l}\in R$ be a linear polynomial defining $H_{i,j}$. 
Then
$$
g_{i,l}=\sum_{j=1}^{n_i} c_{i,j,l}x_{i,j}, \quad l\in\{1,\dots,t_i\},
$$
with $c_{i,j,l}\in\ZZ$, and $g_{i,1},\dots,g_{i,t_i}$ are linearly independent for all $i\in\{1,\dots,s\}$.
Let $m_{i,j}=1$ for all $(i,j)\in\indexset$.
Then 
	$$ F_{\mathbf d}(B_1,\dots,B_s)
	=\prod_{i=1}^sF_{i,d_i}(B_i),
	$$
	where for $i\in\{1,\dots,s\}$, $d\in\ZZ_{>0}$ and $B>0$,
	\begin{multline*}
	F_{i,d}(B)=\sharp\{(x_{i,1},\dots,x_{i,n_i})\in(\ZZ_{\neq 0})^{n_i}: 
	 \sup_{1\leq j\leq n_{i}}|x_{i,j}|\leq B, \\
	 d\mid x_{i,j}  \ \forall j\in\{1,\dots,n_i\}, \
	  g_{i,1}=\dots=g_{i,t_i}=0\}.
	 \end{multline*}

For $i\in\{1,\dots,s\}$, let $W_i\subseteq \RR^{n_i}$ be the linear space defined by $g_{i,1}=\dots=g_{i,t_i}=0$, and let $\Lambda_i\subseteq W_i$ be restriction of the standard lattice $\ZZ^{n_i}\subseteq\RR^{n_i}$ to $W_i$. Then by \cite[Lemma 11.10.15]{BombieriGubler} for every $T\geq 1$,
\begin{gather*}
\sharp (\ZZ^{n_i}\cap[-T,T]^{n_i}\cap W_i)=\sharp (\Lambda_i\cap(T([-1,1]^n\cap W_i)) \\
=T^{n_i-t_i}\frac{\meas_{n_i-t_i}([-1,1]^{n_i}\cap W_i)}{\det\Lambda_i} +O\left(T^{n_i-t_i-1} \right),
\end{gather*}
where $\meas_{n_i-t_i}$ is the $(n_i-t_i)$-dimensional measure induced by the Lebesgue measure on $\RR^{n_i}$. 
Let
$$
c_i=\frac{\meas_{n_i-t_i}([-1,1]^{n_i}\cap W_i)}{\det\Lambda_i}.
$$
Then applying this estimate with $T=B/d$ gives
$$F_{i,d}(B)=c_i(B/d)^{n_i-t_i} +O((B/d)^{n_i-t_i-1})$$ 
whenever $d\leq B$. If $d>B$, then $F_{i,d}(B)=0$ and the same estimate holds.
Hence, for $\delta>0$,
\begin{multline*}
F_{\mathbf d}(B_1,\dots,B_s)=C_{M,\mathbf d} \prod_{i=1}^sB_i^{n_i-t_i} 
+O\left( C_{E,\mathbf d}\left(\prod_{i=1}^sB_i^{n_i-t_i}\right)\left(\min_{1\leq i\leq s}B_i\right)^{-\delta}\right),
\end{multline*}
where
\begin{gather}
\label{eq:linear C_M}
C_{M,\mathbf d}=\prod_{i=1}^s\frac{c_i}{d_i^{n_i-t_i}},\\
C_{E,\mathbf d}=\prod_{i=1}^sd_i^{-(n_i-t_i)+\delta}.
\end{gather}

We show that for $\delta>0$ sufficiently small, the assumptions of Lemma \ref{lem:counting} are satisfied. 
Since $n_i-t_i\geq 2$ for all $i\in\{1,\dots,s\}$ such that $t_i\neq 0$, if $f_{(n_1-t_1,\dots,n_s-t_s)}<2$, by Remark \ref{rem:mu}\ref{item: mu tildef} there is $\tilde i\in\{1,\dots,s\}$ such that $t_{\tilde i}=0$, $\tilde i\in I_\sigma$ for all $\sigma\in\Sigma_{\max}$, and $n_{\tilde i}=1$. Then $\rho_{\tilde i,1}$ is not contained in any maximal cone of $\Sigma$, contradicting the fact that $X$ is proper. Thus
 $f_{(n_1-t_1,\dots,n_s-t_s)}\geq 2$. By definition and by Remark \ref{rem:mu}\ref{item:mu geq2}, $f_{(n_1-t_1-\delta,\dots,n_s-t_s-\delta)}\geq f_{(n_1-t_1,\dots,n_s-t_s)}-s\delta$.

Since $V$ is a smooth complete intersection of smooth divisors, by adjunction \cite[Proposition 16.4]{MR1225842} we have $K_V=K_X+\sum_{i=1}^s\sum_{l=1}^{t_i}[H_{i,l}]$.
Since 
\begin{equation*}
\sum_{i=1}^s(n_i-t_i)\divclass_i=-[K_X]-\sum_{i=1}^st_i\divclass_i=-[K_X]-\sum_{i=1}^s\sum_{l=1}^{t_i}[H_{i,l}],
\end{equation*}
Lemma \ref{lem:counting}  gives
	$$
N_V(B)=cB(\log B)^{b-1} +O(B(\log B)^{b-2}(\log\log B)^s),
$$
where $b=\rk\pic(X)$, and $c$ is defined in \eqref{eq:leading constant} with $k=b-1$, $C_{M,\mathbf d}$ given by \eqref{eq:linear C_M} and $\varpi_i=n_i-t_i$ for $i\in\{1,\dots,s\}$.  The restriction $\pic(X)\to\pic(V)$ is an isomorphism as $t_i\leq n_i-2$ for all $i\in\{1,\dots,s\}$. The leading constant $c$ is positive by Lemma \ref{lem:mobius}\ref{item:mobius positive}.
\end{proof}

\section{Bihomogeneous hypersurfaces}
\label{sec:bihomog}
\begin{proof}[Proof of Theorem \ref{thm:bihomog}]
In the setting of Theorem \ref{thm:bihomog}, the hypersurfaces $H_1,\dots,H_t$ are defined by bihomogeneous polynomials $g_1, \dots,g_t$ of degree $(e_1,e_2)$ in the two sets of variables $\{x_{1,j}:1\leq j\leq n_1\}$ and $\{x_{2,j}:1\leq j\leq n_2\}$. Let $m_{i,j}=1$ for all $(i,j)\in\indexset$.

 We will apply \cite[Theorem 4.4]{Sch} with $R=t$, $F_i=g_i$, $\mathscr B_i=[-1,1]^{n_i}$, $P_i=B_i/d_i$, $d_i=e_i$.
In order to apply  \cite[Theorem 4.4]{Sch} we need to restrict the points to an open set. 	 
	Let $U\subseteq \A^{n_1+n_2}$ be the open set in \cite[Theorem 4.4]{Sch}. Since the complement of $U$ is the zero set of homogeneous polynomials by \cite[Theorems 4.1, 4.2]{Sch}, the set $W:=\pi(\{\mathbf x\in Y: (x_{1,1},\dots, x_{1,n_1},x_{2,1},\dots,x_{2,n_2})\in U\})$ is an open subset of $X$. 
	Then $$N_{V,W}(B)=\frac 1{2^r}\sum_{\mathbf d\in(\ZZ_{>0})^s}\mu(\mathbf d)\sharp A^W(B,\mathbf d),$$ with
	$$A^W(B,\mathbf d)=\sum_{\substack{y_1,\dots,y_s\in\ZZ_{>0}\\ \prod_{i=1}^s y_i^{\alpha_{i,\sigma}}\leq B,\ \forall \sigma\in\Sigma_{\max}}}f^W_{\mathbf d}(y_1,\dots,y_s)$$ and 
	\begin{multline*}
	f^W_{\mathbf d}(y_1,\dots,y_s)=\sharp\{\mathbf x\in(\ZZ_{\neq 0})^{\indexset}: 
	(x_{1,1},\dots,x_{1,n_1}, x_{2,1},\dots,x_{2,n_2}) \in U(\QQ), \\
	\eqref{cond:divisibility},\eqref{cond:m-full},\eqref{cond:equations},
	y_i=\sup_{1\leq j\leq n_{i}}|x_{i,j}| \ \forall i\in\{1,\dots,s\} \}.
	\end{multline*}
	Let $$F^W_{\mathbf d}(B_1,\dots,B_s)=\sum_{1\leq y_i\leq B_i,1\leq i\leq s}f^W_{\mathbf d}(y_1,\dots,y_s).$$
	Then $$ F^W_{\mathbf d}(B_1,\dots,B_s)
	=\widetilde F^W_{d_1,d_2}(B_1,B_2)\prod_{i=3}^sF_{i,d_i}(B_i),
	$$
	where 	
	\begin{multline*}
	\widetilde F_{d_1,d_2}^W=\sharp\{(x_{1,1},\dots,x_{1,n_1}, x_{2,1},\dots,x_{2,n_2})\in(\ZZ_{\neq 0})^{n_1+n_2}\cap U(\QQ): \\
	 \sup_{1\leq j\leq n_{i}}|y_{i,j}|\leq B_i/d_i \ \forall i\in\{1,2\}, \ g_1=\dots=g_t=0
	  \}.
	  \end{multline*}
	 and for $d\in\ZZ_{>0}$ and $B>0$,
	 \begin{multline*}
	F_{i,d}(B)=\sharp\{(x_{i,1},\dots,x_{i,n_i})\in(\ZZ_{\neq 0})^{n_i}: 
	 \sup_{1\leq j\leq n_{i}}|x_{i,j}|\leq B, \\
	 d\mid x_{i,j} \ \forall j\in\{1,\dots,n_i\}
	 \}.
	 \end{multline*}
	 
	  If $d\leq B_i$, then
	 $$
	 F_{i,d}(B)=2^{n_i}(B/d)^{n_i}+O((B/d)^{n_i-\delta})
	 $$
	 with $0<\delta\leq 1$. If $d>B$, then  $F_{i,d}(B)=0$, and the same estimate holds.
	
	To compute $\widetilde F^W_{d_1,d_2}(B_1,B_2)$, write $x_{i,j}=d_iy_{i,j}$ for all $(i,j)\in\indexset$. Then
	\begin{multline*}
	\widetilde F^W_{d_1,d_2}(B_1,B_2)=\sharp\{(y_{1,1},\dots,y_{1,n_1}, y_{2,1},\dots,y_{2,n_2})\in(\ZZ_{\neq 0})^{n_1+n_2}\cap U(\QQ): \\
	 \sup_{1\leq j\leq n_{i}}|y_{i,j}|\leq B_i/d_i \ \forall i\in\{1,2\}, \ g_1=\dots=g_t=0
	  \},
	 \end{multline*}
	 as the complement of $U$ is the zero set of homogeneous polynomials by \cite[Theorems 4.1, 4.2]{Sch}.
	 Let $V_i^*\subseteq\A^{n_1+n_2}$ be the locus where the matrix $\left(\frac{\partial g_l}{\partial x_{i,j}}\right)_{\substack{1\leq l\leq t\\1\leq j\leq n_i}}$ does not have full rank.
	 If $n_1+n_2>\dim V_1^*+\dim V_2^* + 3\cdot 2^{e_1+e_2}e_1e_2t^3$,  
	 then by \cite[Theorem 4.4]{Sch}  there is $\delta>0$ such that
	 \begin{multline*}
	 \widetilde F^W_{d_1,d_2}(B_1,B_2)
	 =C \prod_{i=1}^2\left(B_i/d_i\right)^{n_i-te_i}+O\left( \left(\min_{i=1,2}B_i/d_i\right)^{-\delta}\prod_{i=1}^2(B_i/d_i)^{n_i-te_i}\right)\\
	 =C \prod_{i=1}^2\left(B_i/d_i\right)^{n_i-te_i}+O\left( \left(\prod_{i=1}^2d_i^{-(n_i-te_i)+\delta}\right)\left(\min_{i=1,2}B_i\right)^{-\delta}\prod_{i=1}^2B_i^{n_i-te_i}\right)
	 \end{multline*}	 
	 with $C\in\RR_{\geq 0}$, and $C>0$ whenever $V$ has nonsingular $\QQ_v$-points for all places $v$ of $\QQ$. 
Thus	 
\begin{multline*}
F^W_{\mathbf d}(B_1,\dots,B_s) = C_{M,\mathbf d} B_1^{n_1-te_1}B_2^{n_2-te_2}\prod_{i=3}^sB_i^{n_i} \\+O\left(C_{E,\mathbf d}\left(\min_{1\leq i\leq s}B_i\right)^{-\delta}B_1^{n_1-te_1}B_2^{n_2-te_2}\prod_{i=3}^sB_i^{n_i}\right),
\end{multline*}
	 where 
	 \begin{gather}
	 \label{eq:bihom C_M}
	 C_{M,\mathbf d}=C d_1^{-(n_1-te_1)}d_2^{-(n_2-te_2)} \prod_{i=3}^sd_i^{-n_i},\\
	 C_{E,\mathbf d}= d_1^{-(n_1-te_1)+\delta}d_2^{-(n_2-te_2)+\delta} \prod_{i=3}^sd_i^{-n_i+\delta}.
	 \end{gather}
	 
	 Recall that $n_i - te_i\geq 2$ for $i\in\{1,2\}$. For $\delta>0$ sufficiently small, if $$f_{n_1-te_1-\delta, n_2-te_2-\delta, n_3-\delta,\dots,n_s-\delta}\leq 1,$$ then by Remark \ref{rem:mu}\ref{item: mu tildef} there is $\tilde i\in\{3,\dots,s\}$ such that $\tilde i\in I_\sigma$ for all $\sigma\in\Sigma_{\max}$ and $n_{\tilde i}=1$. Then the ray $\rho_{\tilde i,1}$ is contained in no  maximal cone of $\Sigma$, contradicting the fact that $X$ is proper. 
	 
	Since $V$ is a smooth complete intersection, the adjunction formula \cite[Proposition 16.4]{MR1225842}  gives $K_V=K_X+H_1+\dots+H_t$.	 
	Let $\varpi_i=n_i-te_i$ for $i\in\{1,2\}$, and $\varpi_i=n_i$ for $i\in\{3,\dots,s\}$.
	 Since \begin{equation*}
	\sum_{i=1}^s\varpi_i\divclass_i=-[K_X]-t(e_1\divclass_1+e_2\divclass_2)=-[K_X]-[H_1+\dots+H_t],
	\end{equation*}
	Lemma \ref{lem:counting}  applied to $F^W_{\mathbf d}(B_1,\dots,B_s)$ and $N_{V,W}(B)$ gives 
	$$
N_{V,W}(B)=cB(\log B)^{b-1} +O(B^a(\log B)^{b-2}(\log\log B)^s)
$$
for $B>0$, where $b=\rk\pic(X)$, and $c$ is defined in \eqref{eq:leading constant} with $k=b-1$,  $C_{M,\mathbf d}$ given by \eqref{eq:bihom C_M}.
Moreover, the restriction $\pic(X)\to\pic(V)$ is an isomorphism, as $t\leq\min\{n_1,n_2\}-2$. By Lemma \ref{lem:mobius}\ref{item:mobius positive} the leading constant $c$ is positive if $V(\QQ_v)\neq\emptyset$ for all places $v$ of $\QQ$ as $C$ is positive under the same conditions by \cite[Theorems 4.3 and 4.4]{Sch}.
\end{proof}

\section{Campana points on certain diagonal complete intersections}
\label{sec:Campana}
\begin{proof}[Proof of Theorem \ref{thm:Campana}]

In the setting of Theorem \ref{thm:Campana}, the hypersurfaces $H_1,\dots,H_t$ are defined by homogeneous diagonal polynomials $g_1,\dots,g_t\in R$ with $\deg g_i=e_i\divclass_i$ in $\pic(X)$ for all $i\in\{1,\dots,t\}$.
Then
$$
g_{i}=\sum_{j=1}^{n_i} c_{i,j}x_{i,j}^{e_i}, 
$$
with $c_{i,j}\in\ZZ_{\neq0}$, and 
	$$ F_{\mathbf d}(B_1,\dots,B_s)
	=\prod_{i=1}^sF_{i,d_i}(B_i),
	$$
	where for $i\leq t$,
	\begin{multline*}
	F_{i,d}(B)=\sharp\{(x_{i,1},\dots,x_{i,n_i})\in(\ZZ_{\neq 0})^{n_i}: 
	 d\mid x_{i,j}, \
	 x_{i,j} \text{ is } m_{i,j}\text{-full} \ \forall j\in\{1,\dots,n_i\}, \\
	 \sup_{1\leq j\leq n_{i}}|x_{i,j}|\leq B, \
	  g_i=0\},
	 \end{multline*}
and for $i>t$,
\begin{multline}
	F_{i,d}(B)=\sharp\{(x_{i,1},\dots,x_{i,n_i})\in(\ZZ_{\neq 0})^{n_i}: 
	 \sup_{1\leq j\leq n_{i}}|x_{i,j}|\leq B, \
	 d\mid x_{i,j}, \\ x_{i,j} \text{ is } m_{i,j}\text{-full} \ \forall j\in\{1,\dots,n_i\}\}.
	 \label{eq:campana coord}
	 \end{multline}

For $i\leq t$, we estimate $F_{i,d}(B)$ via the following lemma.

\begin{lemma}\label{lem:diagonal}
Let $n,e, m_1,\dots,m_n\in\ZZ_{>0}$. Let $c_1,\dots,c_n\in\ZZ_{\neq 0}$. Let $d$ be a squarefree positive integer. 
Assume that $n\geq 2$ and $2\leq m_1\leq \dots\leq m_n$.
\begin{enumerate}
\item If $e=1$, assume that $\sum_{j=1}^n\frac 1{m_j}>3$, and $\sum_{j=1}^{n-1}\frac 1{em_j(em_j+1)}\geq 1$. 
\item If $e\geq 2$ assume that $\sum_{j=1}^n\frac 1{em_j}>3$, $\sum_{j=1}^n\frac 1{2s_0(em_j)}>1$, where
$$
s_0(m)=\min\left\{2^{m-1},\frac 12m(m-1)+\lfloor\sqrt{2m+2}\rfloor\right\}, \ m\in\ZZ_{\geq0}.
$$
\end{enumerate}
For $B>0$, let
\begin{multline*}F_{d}(B)=\sharp\{(x_{1},\dots,x_{n})\in(\ZZ_{\neq 0})^{n}: 
	 d\mid x_{j}, \ 
	 x_{j} \text{ is } m_{j}\text{-full} \ \forall j\in\{1,\dots,n\}, \\
	 \sup_{1\leq j\leq n}|x_{j}|\leq B, \
	  \sum_{j=1}^{n} c_{j}x_{j}^{e}=0\}.
\end{multline*}
Then there is $\eta>0$ such that 
$$
F_{d}(B)=c_{e,d}B^\Gamma +O(d^{-1-\eta}B^{\Gamma-\eta}),
$$
where $\Gamma=\sum_{j=1}^n\frac 1{m_j}-e$ and $c_{e,d}$ is defined in \eqref{eq:ced} and satisfies $0\leq c_{e,d}\ll d^{-1-\eta}$.
\end{lemma}
\begin{proof}
For every $j\in\{1,\dots,n\}$ and $x_j\in\ZZ_{\neq 0}$ that is $m_j$-full,
there exist unique $u_j,v_{j,1},\dots,v_{j,m_j-1}\in\ZZ_{>0}$ such that 
\begin{gather*}
|x_j|=u_j^{m_j}\prod_{r=1}^{m_j-1}v_{j,r}^{m_j+r}, \\ \mu^2(v_{j,r})=1, \quad \gcd(v_{j,r},v_{j,r'})=1 \quad \forall r, r'\in\{1,\dots,m_j-1\}, r\neq r'.
\end{gather*}
For every choice of $u_j$ and $v_{j,r}$ as above, if $d\mid x_j$ with $d\in\ZZ_{> 0}$ squarefree, then there exist unique $s_j,t_{j,1},\dots,t_{j,m_j-1}\in\ZZ_{>0}$ such that 
\begin{gather*}
d=s_j\prod_{r=1}^{m_j-1}t_{j,r}, \quad
\mu^2(s_j)=\mu^2(t_{j,r})=1  \ \forall r\in\{1,\dots,m_j-1\} \\
\gcd(s_j,v_{j,r})=\gcd(s_j,t_{j,r})=\gcd(t_{j,r},t_{j,r'})=1 \ \forall r,r'\in\{1,\dots,m_j-1\}, r\neq r'\\
s_j\mid u_j, \ t_{j,r}\mid v_{j,r} \ \forall r\in\{1,\dots,m_j-1\}.
\end{gather*}
Write $u_j=s_j\tilde u_j$ and $v_{j,r}=t_{j,r}\tilde v_{j,r}$ for all $r\in\{1,\dots,m_{j-1}\}$. Write $\mathbf s=(s_1,\dots,s_n)$, $\mathbf t=(t_{j,r})_{1\leq j\leq n, 1\leq r\leq m_j-1}$. For $j\in\{1,\dots,n\}$, write 
$$\sigma_j=s_j\prod_{r=1}^{m_j-1}t_{j,r}, \qquad \tau_j= s_j^{m_j}\prod_{r=1}^{m_j-1}t_{j,r}^{m_j+r}, \qquad w_j=\prod_{r=1}^{m_j-1}\tilde v_{j,r}^{m_j+r}.$$

Let $\mathcal T_d(B)$ be the set of pairs $(\mathbf s,\mathbf t)\in\ZZ_{>0}^n\times\ZZ_{>0}^{\sum_{j=1}^n(m_j-1)}$ that satisfy
\begin{gather*}
\mu^2(\sigma_j)=1, \quad d=\sigma_j, \quad \tau_j\leq B
\quad
 \forall j\in\{1,\dots,s\}.
\end{gather*}
Note that the first two conditions imply 
\begin{equation}\label{eq:card T_d(B)}
\sharp \mathcal T_d(B) \leq \prod_{j=1}^n m_j^{\omega(d)} \ll d^{\epsilon},
\end{equation}
where $\omega(d)$ is the number of distinct prime divisors of $d$.
Let $\mathcal V_{\mathbf s,\mathbf t}(B)$ be the set of $\tilde{\mathbf v}=(\tilde v_{j,r})_{1\leq j\leq n,1\leq r\leq m_j-1}\in\ZZ_{>0}^{\sum_{j=1}^n(m_j-1)}$ such that
\begin{gather*}
\mu^2(t_{j,r}\tilde v_{j,r})=1, \ \gcd(s_j, \tilde v_{j,r})=1  \ \forall j\in\{1,\dots,n\},  r\in\{1,\dots,m_j-1\},\\
\gcd(t_{j,r}\tilde v_{j,r},t_{j,r'}\tilde v_{j,r'})=1  \ \forall j\in\{1,\dots,n\},  r,r'\in\{1,\dots,m_j-1\}, r\neq r',\\
\tau_j w_j\leq B
\ \forall j\in\{1,\dots,n\}.
\end{gather*}
Let $\mathcal T_d(\infty)=\bigcup_{B>0}\mathcal T_d(B)$ and $\mathcal V_{\mathbf s,\mathbf t}(\infty)=\bigcup_{B>0}\mathcal V_{\mathbf s,\mathbf t}(B)$.

Then
\begin{equation}\label{eq:cases}
F_{d}(B)=\begin{cases}
\sum_{\beps\in\{\pm1\}^n}\sum_{(\mathbf s,\mathbf t)\in \mathcal T_d(B)}\sum_{\tilde{\mathbf v}\in\mathcal V_{\mathbf s,\mathbf t}(B)} M_{\beps\mathbf c,\bgamma}(B^e) &  \text{ if $e$ is odd}, \\
2^n\sum_{(\mathbf s,\mathbf t)\in \mathcal T_d(B)}\sum_{\tilde{\mathbf v}\in\mathcal V_{\mathbf s,\mathbf t}(B)} M_{\mathbf c,\bgamma}(B^e) & \text{ if $e$ is even},\\
\end{cases}
\end{equation}
where 
$\mathbf c=(c_1,\dots,c_n)$, $\beps=(\varepsilon_1,\dots,\varepsilon_n)$, $\beps\mathbf c=(\varepsilon_1c_1,\dots,\varepsilon_nc_n)$,
$\bgamma=(\gamma_1,\dots,\gamma_n)$ with
$$
\gamma_j=s_j^{em_j}\prod_{r=1}^{m_j-1}t_{j,r}^{e(m_j+r)}\tilde v_{j,r}^{e(m_j+r)} \quad  \forall j\in\{1,\dots,n\},
$$
and
$$
M_{\beps\mathbf c,\bgamma}(B^e)=\sharp\left\{(\tilde u_1,\dots,\tilde u_n)\in\ZZ_{>0}^n: \max_{1\leq j\leq n}\gamma_j\tilde u_j^{em_j}\leq B^e, \ \sum_{j=1}^n\varepsilon_j c_j\gamma_j\tilde u_j^{em_j}=0\right\}.
$$

An estimate for $M_{\beps\mathbf c,\bgamma}(B^e)$ is proven in \cite[Theorem 2.7]{BY} in the case where 
$\sum_{j=1}^{n-1}\frac 1{em_j(em_j+1)}\geq 1$.
The subsequent paper \cite[Theorem 5.3]{BBKOPWproc} extends the range of applicability of \cite[Theorem 2.7]{BY} to the case where  
$\sum_{j=1}^n\frac 1{em_j}>3$, $\sum_{j=1}^n\frac 1{2s_0(em_j)}>1$.

Let 
$$
\Theta_e = \begin{cases}\frac 1{m_n(m_n+1)} & \text{ if } e=1,\\
\sum_{j=1}^n\frac 1{2s_0(em_j)}-1 & \text{ if } e\geq 2. \\ \end{cases}
$$
For $0<\delta<\frac 1{(2(n-1)+5)em_n(em_n+1)}$ and $\epsilon>0$, \cite[Theorem 2.7]{BY} and \cite[Theorem 5.3]{BBKOPWproc} give
\begin{multline}\label{eq:BBKOPW}
\sum_{(\mathbf s,\mathbf t)\in \mathcal T_d(B)}\sum_{\tilde{\mathbf v}\in\mathcal V_{\mathbf s,\mathbf t}(B)} M_{\beps\mathbf c,\bgamma}(B^e)
=\sum_{(\mathbf s,\mathbf t)\in \mathcal T_d(B)}\sum_{\tilde{\mathbf v}\in\mathcal V_{\mathbf s,\mathbf t}(B)} \frac{\mathfrak S_{\beps\mathbf c,\mathbf \bgamma}\mathfrak I_{\beps\mathbf c}}{\prod_{j=1}^n\gamma_j^{\frac 1{em_j}}}B^{\Gamma}\\
+O\left(B^{\Gamma}(F_1+F_2+F_3)\right),
\end{multline}
where 
\begin{gather*}
\mathfrak S_{\beps\mathbf c,\bgamma}= \sum_{q=1}^\infty \frac 1{q^n}\sum_{\substack{ a (\text{mod } q)\\ \gcd(a,q)=1}}\prod_{j=1}^n\sum_{r=1}^q \exp(2\pi i a\varepsilon_jc_j\gamma_jr^{em_j}/q), \\
\mathfrak I_{\beps\mathbf c}=\int_{-\infty}^\infty \prod_{j=1}^n\left(\int_0^1\exp(2\pi i\lambda \varepsilon_jc_j\xi^{em_j})\mathrm d\xi\right)\mathrm d\lambda,\\
F_1=B^{e((2(n-1)+5)\delta-1)-\Gamma} 
\sum_{(\mathbf s,\mathbf t)\in \mathcal T_d(B)}\sum_{\tilde{\mathbf v}\in\mathcal V_{\mathbf s,\mathbf t}(B)}
\left( \prod_{j=1}^n \frac{B^{\frac 1{m_j}}}{\gamma_j^{\frac1{em_j}}} \right)
\sum_{l=1}^n\frac{\gamma_l^{\frac1{em_l}}}{B^{\frac 1{m_l}}}, \\
F_2=B^{-e\delta}\sum_{(\mathbf s,\mathbf t)\in \mathcal T_d(B)}\sum_{\tilde{\mathbf v}\in\mathcal V_{\mathbf s,\mathbf t}(B)} \sum_{q=1}^\infty q^{1-\Gamma/e+\epsilon}\prod_{j=1}^n\gcd(\gamma_j,q)^{\frac 1{em_j}}\gamma_j^{-\frac 1{em_j}}, \\
F_3=\begin{cases} B^{-e\delta\Theta_e+\epsilon} \sum_{(\mathbf s,\mathbf t)\in \mathcal T_d(B)}\sum_{\tilde{\mathbf v}\in\mathcal V_{\mathbf s,\mathbf t}(B)} \prod_{j=1}^n\gamma_j^{-\frac 1{m_j+1}} &\text{ if }e=1,\\
B^{-e\delta\Theta_e+\epsilon} \sum_{(\mathbf s,\mathbf t)\in \mathcal T_d(B)}\sum_{\tilde{\mathbf v}\in\mathcal V_{\mathbf s,\mathbf t}(B)} \prod_{j=1}^n\gamma_j^{-\frac 1{em_j}+\frac 1{2s_0(em_j)}} & \text{ if } e\geq 2.
\end{cases}
\end{gather*}

Let 
\begin{gather}\label{eq:ced}
c_{e,d}=\begin{cases}
\sum_{\beps\in\{\pm1\}^n}\sum_{(\mathbf s,\mathbf t)\in \mathcal T_d(\infty)}\sum_{\tilde{\mathbf v}\in\mathcal V_{\mathbf s,\mathbf t}(\infty)} \frac{\mathfrak S_{\beps\mathbf c,\mathbf \bgamma}\mathfrak I_{\beps\mathbf c}}{\prod_{j=1}^n\gamma_j^{\frac 1{em_j}}} &  \text{ if $e$ is odd}, \\
2^n\sum_{(\mathbf s,\mathbf t)\in \mathcal T_d(\infty)}\sum_{\tilde{\mathbf v}\in\mathcal V_{\mathbf s,\mathbf t}(\infty)} \frac{\mathfrak S_{\mathbf c,\mathbf \bgamma}\mathfrak I_{\mathbf c}}{\prod_{j=1}^n\gamma_j^{\frac 1{em_j}}} & \text{ if $e$ is even}.\\
\end{cases} 
\end{gather}

For $T>0$, let
\begin{gather*}
f_1(q)=\sum_{(\mathbf s,\mathbf t)\in \mathcal T_d(\infty)}\prod_{j=1}^n\left(\frac{\gcd(\tau_j^e,q)}{\tau_j^e}\right)^{\frac 1{em_j}}, \\
f_2(q)=\sum_{\tilde{\mathbf v}\in \mathcal V_{\mathbf{1},\mathbf{1}}(\infty)}\prod_{j=1}^n\left(\frac{\gcd(w_j^e,q)}{w_j^e}\right)^{\frac 1{em_j}},
\end{gather*}
and
\begin{gather*}
f_2(q,T,\mathbf{s},\mathbf{t})=\sum_{\tilde{\mathbf v}\in \mathcal V_{\mathbf s,\mathbf t}(\infty)\smallsetminus \mathcal V_{\mathbf s,\mathbf t}(T)}\prod_{j=1}^n\left(\frac{\gcd(w_j^e,q)}{w_j^e}\right)^{\frac 1{em_j}}.
\end{gather*}

Note that for $\sum_{j=1}^n \frac{1}{em_j} >1$ we have
$$\left| \mathfrak I_{\beps\mathbf c}\right| \ll 1.$$
Similarly as in \cite[(2.8), (2.9), (2.12)]{BY}, the difference between $c_{e,d}B^{\Gamma}$ and the main term obtained combining \eqref{eq:cases} and \eqref{eq:BBKOPW} is bounded by 

\begin{gather}
B^{\Gamma}\sum_{(\mathbf s,\mathbf t)\in \mathcal T_d(\infty)} \sum_{\tilde{\mathbf v}\in \mathcal V_{\mathbf s,\mathbf t}(\infty)\smallsetminus \mathcal V_{\mathbf s,\mathbf t}(B)} \sum_{q=1}^\infty q^{1-\sum_{j=1}^n \frac{1}{em_j}} \prod_{j=1}^n \gamma_j^{-\frac{1}{em_j}} \gcd(\gamma_j,q)^{\frac{1}{em_j}}\\\label{eq:ced_difference}
\ll B^{\Gamma}\sum_{q=1}^\infty q^{-\Gamma/e+\epsilon} \sum_{(\mathbf s,\mathbf t)\in \mathcal T_d(\infty)}\prod_{j=1}^n\left(\frac{\gcd(\tau_j^e,q)}{\tau_j^e}\right)^{\frac 1{em_j}} f_2(q,B,\mathbf{s},\mathbf{t}),
\end{gather}
and $c_{e,d}\ll \sum_{q=1}^\infty q^{-\Gamma/e+\epsilon}f_1(q)f_2(q)$.

By \cite[(3.10)]{BY} and the arguments used to prove \cite[(3.9)]{BY}, we have
\begin{equation}\label{f2}
f_2(q)\ll q^{\epsilon},
\end{equation}
and
\begin{gather*}
f_2(q,T,\mathbf{s},\mathbf{t})\ll \sum_{i_0=1}^n \sum_{\substack{\tilde v_{i,r},\ 1\leq i\leq n,\ 1\leq r\leq m_i-1\\\prod_{r=1}^{m_i-1}\tilde v_{i,r}^{m_i+r}>\frac{T}{\tau_i}\ \text{if } i=i_0}} \prod_{i=1}^n\prod_{r=1}^{m_i-1}\frac{\mu^2(\tilde v_{i,r})\gcd(\tilde v_{i,r}^{e(m_i+r)},q)^\frac 1{em_i}}{\tilde v_{i,r}^{(m_i+r)/m_i}}\\
\ll q^{\epsilon} \sum_{i=1}^n \sum_{\substack{\tilde v_{i,r},\ 1\leq r\leq m_i-1\\\prod_{r=1}^{m_i-1}\tilde v_{i,r}^{m_i+r}>\frac{T}{\tau_i} }} \prod_{r=1}^{m_i-1}\frac{\mu^2(\tilde v_{i,r})\gcd(\tilde v_{i,r}^{e(m_i+r)},q)^\frac 1{em_i}}{\tilde v_{i,r}^{(m_i+r)/m_i}}.
\end{gather*}

Our next goal is to provide an upper bound for sums of the type occurring in this estimate for $f_2(q,T, \mathbf{s},\mathbf{t})$.

\begin{lemma}\label{lemdecay}
Let $m\in \mathbb{N}_{\geq 2}$, $e\in \mathbb{N}$ and $A>0$ a real parameter. Then, for every $0<\epsilon <\frac{1}{m(m+1)}$ we have
\begin{equation*}
 \sum_{\substack{ v_{r}\in \mathbb{N},\ 1\leq r\leq m-1\\\prod_{r=1}^{m-1} v_{r}^{m+r}>A }} \prod_{r=1}^{m-1}\frac{\mu^2( v_{r})\gcd( v_{r}^{e(m+r)},q)^\frac 1{em}}{ v_{r}^{(m+r)/m}} \ll_{m,\epsilon} A^{-\frac{1}{m(m+1)}+\epsilon} q^{\frac{m-1}{em(m+1)}+\epsilon}.
\end{equation*}

\end{lemma}

\begin{proof}
We first consider the sum
$$S_1:=  \sum_{\substack{ v_{r}\in \mathbb{N},\ 1\leq r\leq m-1\\\prod_{r=1}^{m-1} v_{r}^{m+r}>A }} \prod_{r=1}^{m-1}\frac{1}{ v_{r}^{(m+r)/m}}$$
for $A>1$. A dyadic decomposition for each of the variables $v_r$, $1\leq r\leq m-1$, leads us to the upper bound
\begin{gather*}
S_1\ll \sum_{\substack{l_1,\ldots, l_{m-1}\in \mathbb{N}\\ 2^{(m+1)l_1+\ldots + (2m-1)l_{m-1}}>A}} 2^{-\frac{1}{m}l_1-\ldots - \frac{m-1}{m} l_{m-1}} 
\end{gather*}
Note that for each $k\in \frac{1}{m}\mathbb{N}$ we have
$$\sharp\{l_1,\ldots, l_{m-1}\in \mathbb{N}:\ \frac{1}{m}l_1+\ldots + \frac{m-1}{m} l_{m-1}=k\}\ll_m k^{m-2}.$$
We deduce that
\begin{gather*}
S_1\ll_m \sum_{\substack{k\in \frac{1}{m}\mathbb{N}\\ r(k)>0}} k^{m-2} 2^{-k} 
\end{gather*}
where $r(k)$ is the number of $(l_1,\ldots, l_{m-1})\in \mathbb{N}^{m-1}$ such that both $$\frac{1}{m}l_1+\ldots + \frac{m-1}{m} l_{m-1}=k \quad \mbox{and}\quad 2^{(m+1)l_1+\ldots + (2m-1)l_{m-1}}>A$$ hold. Observe that if $r(k)>0$, then there exists $(l_1,\ldots, l_{m-1})\in \mathbb{N}^{m-1}$ with $\frac{1}{m}l_1+\ldots + \frac{m-1}{m} l_{m-1}=k $ and
\begin{gather*}
m(m+1)k=(m+1) \left(l_1+\ldots + (m-1) l_{m-1}\right)\\  \geq (m+1) l_1+ \frac{m+2}{2} 2 l_2+\ldots + \frac{2m-1}{m-1} (m-1)l_{m-1}> \log A/\log 2,
\end{gather*}
i.e. 
$$S_1\ll_m \sum_{\substack{ k\in \frac{1}{m}\mathbb{N}\\m(m+1)k>\log A/\log 2}} k^{m-2} 2^{-k}\ll_{m,\epsilon} A^{-\frac{1}{m(m+1)}+\epsilon}.$$
Note that the upper bound for $S_1$ also holds for $A\leq1$ and $\epsilon < \frac{1}{m(m+1)}$.\par

We now turn to the sum in the statement of the lemma. If $v_r$ is a square-free natural number and $d_r=\gcd(v_r^{e(m+r)},q)$, then we can write 
$$d_r=d_{r,1}d_{r,2}^2\ldots d_{r,e(m+r)}^{e(m+r)},\quad \mu^2(d_{r,j})=1,\ \forall 1\leq j\leq e(m+r),\quad \gcd(d_{r,j},d_{r,j'})=1,\ \forall j\neq j'.$$
Writing $v_r=v_r'\prod_{j=1}^{e(m+r)}d_{r,j}$ and $d_r'=\prod_{j=1}^{e(m+r)}d_{r,j}$ we find that
\begin{gather*}
S_2:=\sum_{\substack{ v_{r}\in \mathbb{N},\ 1\leq r\leq m-1\\\prod_{r=1}^{m-1} v_{r}^{m+r}>A }} \prod_{r=1}^{m-1}\frac{\mu^2( v_{r})\gcd( v_{r}^{e(m+r)},q)^\frac 1{em}}{ v_{r}^{(m+r)/m}} \\
\ll \sum_{d_{r,1}d_{r,2}^2\ldots d_{r,e(m+r)}^{e(m+r)} \mid q,\ 1\leq r\leq m-1} \sum_{\substack{ v'_{r}\in \mathbb{N},\ 1\leq r\leq m-1\\\prod_{r=1}^{m-1} (d'_rv'_{r})^{m+r}>A }} \prod_{r=1}^{m-1}\frac{d_r^\frac 1{em}}{ (d'_rv'_{r})^{(m+r)/m}} \\
\ll \sum_{d_{r,1}d_{r,2}^2\ldots d_{r,e(m+r)}^{e(m+r)} \mid q,\ 1\leq r\leq m-1} \prod_{r=1}^{m-1} \left(\frac{d_r^{\frac{1}{em}}}{d_r'^{\frac{m+r}{m}}}\right) \sum_{\substack{ v'_{r}\in \mathbb{N},\ 1\leq r\leq m-1\\\prod_{r=1}^{m-1} (d'_rv'_{r})^{m+r}>A }} \prod_{r=1}^{m-1}\frac{1}{ (v'_{r})^{(m+r)/m}}
\end{gather*}
By using the upper bound for $S_1$ we find that for $\epsilon >0$ sufficiently small
\begin{gather*}
S_2 \ll_{\epsilon,m} \sum_{d_{r,1}d_{r,2}^2\ldots d_{r,e(m+r)}^{e(m+r)} \mid q,\ 1\leq r\leq m-1} \prod_{r=1}^{m-1} \left(\frac{d_r^{\frac{1}{em}}}{d_r'^{\frac{m+r}{m}}}\right) A^{-\frac{1}{m(m+1)}+\epsilon} \left(\prod_{r=1}^{m-1}(d_r')^{m+r}\right)^{\frac{1}{m(m+1)}} \\
\ll_{\epsilon,m} A^{-\frac{1}{m(m+1)}+\epsilon} \sum_{d_{r,1}d_{r,2}^2\ldots d_{r,e(m+r)}^{e(m+r)} \mid q,\ 1\leq r\leq m-1} \prod_{r=1}^{m-1} \left( d_r^{\frac{1}{em}} (d_r')^{-\frac{m+r}{m+1}}\right)\\
\ll_{\epsilon,m} A^{-\frac{1}{m(m+1)}+\epsilon} \sum_{d_{r,1}d_{r,2}^2\ldots d_{r,e(m+r)}^{e(m+r)} \mid q,\ 1\leq r\leq m-1} \prod_{r=1}^{m-1}  d_{r}^{\frac{1}{em}-\frac{m+r}{e(m+r)(m+1)}}\\
\ll_{\epsilon,m} A^{-\frac{1}{m(m+1)}+\epsilon} \sum_{d_r \mid q,\ 1\leq r\leq m-1} \prod_{r=1}^{m-1} d_{r}^{\frac{1}{em(m+1)}} \ll_{\epsilon,m} A^{-\frac{1}{m(m+1)}+\epsilon} q^{\frac{m-1}{em(m+1)}+\epsilon}.\qedhere
\end{gather*}

\end{proof}

Lemma \ref{lemdecay} shows that we can bound $f_2(q,T,\mathbf{s},\mathbf{t})$ by
\begin{gather*}
f_2(q,T,\mathbf{s},\mathbf{t})\ll \sum_{i=1}^n \left(\frac{T}{\tau_i}\right)^{-\frac{1}{m_i(m_i+1)}+\epsilon} q^{\frac{m_i-1}{em_i(m_i+1)}+\epsilon}.
\end{gather*}

In the following we write $\Delta_i= \frac{m_i-1}{em_i(m_i+1)}$. Then \eqref{eq:ced_difference} is bounded by 

\begin{gather*}
S_3 := B^{\Gamma} \sum_{i=1}^n  \sum_{q=1}^\infty q^{-\Gamma/e+\Delta_i+\epsilon} \sum_{(\mathbf s,\mathbf t)\in \mathcal T_d(\infty)}\prod_{j=1}^n\left(\frac{\gcd(\tau_j^e,q)}{\tau_j^e}\right)^{\frac 1{em_j}}  \left(\frac{B}{\tau_i}\right)^{-\frac{1}{m_i(m_i+1)}+\epsilon}\\
\ll B^{\Gamma} \sum_{i=1}^n B^{-\frac{1}{m_i(m_i+1)}+\epsilon} \sum_{(\mathbf s,\mathbf t)\in \mathcal T_d(\infty)}\sum_{q=1}^\infty q^{-\Gamma/e+\Delta_i+\epsilon} \tau_i^{\frac{1}{m_i(m_i+1)}}\prod_{j=1}^n\left(\frac{\gcd(\tau_j^e,q)}{\tau_j^e}\right)^{\frac 1{em_j}}.  
\end{gather*} 

As we will encounter similar expressions in our further analysis, we introduce for $E,D>0$ and $d$ squarefree the following sum

$$S_d(D,E):=d^E \sum_{(\mathbf s,\mathbf t)\in \mathcal T_d(\infty)}\sum_{q=1}^\infty q^{-\Gamma/e+D+\epsilon} \prod_{j=1}^n\left(\frac{\gcd(\tau_j^e,q)}{\tau_j^e}\right)^{\frac 1{em_j}}. $$

We write $q=q_1q_2$ with $\gcd(q_1,d)=1$ and such that all prime divisors of $q_2$ divide $d$. With this we obtain that
\begin{gather*}
S_d(D,E)\ll d^E \sum_{(\mathbf s,\mathbf t)\in \mathcal T_d(\infty)} \sum_{q_1=1}^\infty q_1^{-\Gamma/e+D+\epsilon} \sum_{\substack{q_2=1\\p\mid q_2\Rightarrow p\mid d}}^\infty q_2^{-\Gamma/e+D+\epsilon}  \prod_{j=1}^n\left(\frac{\gcd(\tau_j^e,q_2)}{\tau_j^e}\right)^{\frac 1{em_j}}.  
\end{gather*}

If we assume $-\Gamma/e+D<-1$, then the sum over $q_1$ is absolutely convergent. For a given vector $(\mathbf{s},\mathbf{t})\in \mathcal T_d(\infty)$ and a prime $p$ we write $\tau_{j,p}$ for the power of $p$ which exactly divides $\tau_j$. We find that
\begin{gather*}
S_d(D,E)\ll  \sum_{(\mathbf s,\mathbf t)\in \mathcal T_d(\infty)}   d^E \prod_{p\mid d}\left( \sum_{l=0}^\infty p^{l(-\Gamma/e+D+\epsilon)}  \prod_{j=1}^n\left(\frac{\gcd(\tau_{j,p}^e,p^l)}{\tau_{j,p}^e}\right)^{\frac 1{em_j}}  \right).
\end{gather*}
We now split the summation over $l$ into the term $l=0$, where we use the inequality $\tau_{j,p}\geq p^{m_j}$, and we bound the rest by a geometric sum for $l\geq 1$ using $\gcd(\tau_{j,p}^e,p^l)\leq\tau_{j,p}^e$,
\begin{gather*}
S_d(D,E) \ll_D  \sum_{(\mathbf s,\mathbf t)\in \mathcal T_d(\infty)}  d^{E+\epsilon}  \prod_{p\mid d}\left( p^{-n} + p^{-\Gamma/e+D+\epsilon}\right)\\
\ll_D  d^{\epsilon}  \prod_{p\mid d}\left( p^{E-n} + p^{-\Gamma/e+D+E+\epsilon}\right)
\end{gather*}
If $-\Gamma/e+D+E<-1$, then we deduce that
\begin{equation}\label{SdDE}
S_d(D,E)\ll_D d^{-1-\eta}
\end{equation} 
for some $\eta >0$.\par
Applying equation (\ref{SdDE}) to $S_3$ with $D=\Delta_i$ and $E=\frac{2m_i-1}{m_i(m_i+1)}$ we obtain $S_3\ll B^{\Gamma-\eta} d^{-1-\eta}$, for some $\eta>0$. 
Hence,
$$
F_{d}(B)=c_{e,d}B^\Gamma +O(B^\Gamma(d^{-1-\eta}B^{-\eta} +F_1+F_2+F_3)).
$$
We use the bound in (\ref{f2}) and apply equation (\ref{SdDE}) with $D=E=0$ to get $c_{e,d}\ll d^{-1-\eta}$.\par
It remains to estimate the error terms $F_1,F_2,F_3$. We rewrite $F_1$ as follows:
$$
F_1=B^{e\delta(2(n-1)+5)} 
\sum_{l=1}^nB^{-\frac 1{m_l}} \sum_{(\mathbf s,\mathbf t)\in \mathcal T_d(B)}\sum_{\tilde{\mathbf v}\in\mathcal V_{\mathbf s,\mathbf t}(B)}
\prod_{\substack{1\leq j\leq n\\j\neq l}}\gamma_j^{-\frac 1{em_j}}.
$$
 As in \cite[\S3]{BY} and \cite[\S6]{BBKOPWproc}, we have
\begin{gather*}
F_1\ll B^{-\frac 1{m_n(m_n+1)} + e\delta(2(n-1)+5)}\sum_{(\mathbf s,\mathbf t)\in \mathcal T_d(B)} \prod_{j=1}^n\tau_j^{-\frac 1{m_j+1}} \\
\ll d^{-\sum_{j=1}^n\frac{m_j}{m_j+1}+\varepsilon} B^{-\frac 1{m_n(m_n+1)} + e\delta(2(n-1)+5)},
\end{gather*}
where the last estimate follows from
\begin{gather*}
\sum_{(\mathbf s,\mathbf t)\in \mathcal T_d(B)} \prod_{j=1}^n\tau_j^{-\frac 1{m_j+1}} 
\leq  \sum_{(\mathbf s,\mathbf t)\in \mathcal T_d(B)} \prod_{j=1}^n\sigma_j^{-\frac {m_j}{m_j+1}}\\
\leq  d^{-\sum_{j=1}^n\frac{m_j}{m_j+1}}\sharp \mathcal T_d(B)
\ll  d^{-\sum_{j=1}^n\frac{m_j}{m_j+1}+\varepsilon}.
\end{gather*}
by \eqref{eq:card T_d(B)}.
Combining the arguments for $F_3$ in \cite[\S3]{BY} and in \cite[\S6]{BBKOPWproc} and the estimate above we have
\begin{gather*}
F_3\ll B^{-e\delta\Theta_e +\epsilon} \sum_{(\mathbf s,\mathbf t)\in \mathcal T_d(B)} \prod_{j=1}^n\tau_j^{-\frac 1{m_j+1}} \ll d^{-\sum_{j=1}^n\frac{m_j}{m_j+1}+\epsilon} B^{-e\delta\Theta_e +\epsilon}.
\end{gather*}
Since $\sum_{j=1}^n\frac{m_j}{m_j+1}\geq \frac 23n>1$ is satisfied for $n\geq 2$, we have $F_1,F_3\ll d^{-1-\eta} B^{-\eta}$ for a suitable $\eta>0$.
Since
$$
F_2=B^{-e\delta}\sum_{q=1}^\infty q^{1-\Gamma/e+\epsilon}f_1(q)f_2(q),
$$
the estimate (\ref{f2}) combined with (\ref{SdDE}) for $D=1$ and $E=0$ yields $F_2\ll d^{-1-\eta}B^{-e\delta}$, as $\Gamma/e >2$.\end{proof}

By Lemma \ref{lem:diagonal} and \cite[Lemma 5.6]{hypmethod}
	$$ F_{\mathbf d}(B_1,\dots,B_s)=\prod_{i=1}^s \left(c_{M,i}B_i^{\varpi_i} + O(d_i^{\nu_i+\varepsilon}B_i^{\varpi_i-\delta})\right)$$
	where 
	\begin{gather}\label{eq:Campana varpi}
	\varpi_i=\begin{cases}\sum_{j=1}^{n_i}\frac 1{m_{i,j}}-e_i & \text{ if } i\leq t, \\ \sum_{j=1}^{n_i}\frac 1{m_{i,j}} & \text{ if } i>t,\\\end{cases}
	\end{gather}
	 $\nu_i<-1$ for $i\leq t$, $\nu_i=-\frac 23 n_i$ if $i>t$, $c_{M,i}$ is the constant $c_{e_i,d_i}$ defined in \eqref{eq:ced} if $i\leq t$, and $c_{M,i}=2^{n_i}\left(\prod_{j=1}^{n_i}c_{m_{i,j},d_i}\right)$, where $c_{m_{i,j},d_i}$ is the constant defined in \cite[(5.11)]{hypmethod}.
	
	Thus 
	$$
	F_{\mathbf d}(B_1,\dots,B_s)=C_{M,\mathbf d}\prod_{i=1}^sB_i^{\varpi_i}+O\left(C_{E,\mathbf d}(\min_{1\leq i\leq s}B_i)^{-\delta}\prod_{i=1}^sB_i^{\varpi_i}\right),
	$$
	where 
	\begin{equation}\label{eq:Campana C_M}
	C_{M,\mathbf d}=\prod_{i=1}^sc_{M,i}.
	\end{equation}
	Lemma \ref{lem:diagonal} and \cite[(5.14), (5.15)]{hypmethod} give $C_{M,\mathbf d}, C_{E,\mathbf d}\ll \prod_{i=1}^sd_i^{-\beta_i}$ with $\beta_i>1$ whenever $n_i\geq 2$, and $\beta_i>\frac 23-\varepsilon$ otherwise. For $\varepsilon>0$ sufficiently small, $\beta_i+\beta_j>1$ for every $i,j\in\{1,\dots,s\}$. Thus, if $f_{\beta_1,\dots,\beta_s}\leq 1$, by Remark \ref{rem:mu}\ref{item: mu tildef} then there exists an index $\tilde i\in\{1,\dots,s\}$ such that $\tilde i\in I_\sigma$ for all $\sigma\in\Sigma_{\max}$ and $n_{\tilde i}=1$. Then the ray $\rho_{\tilde i,1}$ is contained in no  maximal cone of $\Sigma$, contradicting the fact that $X$ is proper. 
	
Since $c_{i,j}\neq 0$ for all $i\in\{1,\dots,t\}$, $j\in\{1,\dots,n_i\}$, the adjunction formula \cite[Proposition 16.4]{MR1225842} gives $K_V=(K_X+ H_1+\dots+H_t)|_V$.
	Since
	\begin{equation*}
\sum_{i=1}^s\varpi_i\divclass_i=-K_X-\sum_{i=1}^s\sum_{j=1}^{n_i}(1-\frac 1{m_{i,j}})\divclass_i +\sum_{i=1}^te_i\divclass_i=-(K_X+[\mathscr D_{\mathbf m}|_X] + [H_1+\dots+H_t]),
\end{equation*}
Lemma \ref{lem:counting} gives
$$
N_V(B)=cB(\log B)^{b-1} +O(B(\log B)^{b-2}(\log\log B)^s),
$$
where $b=\rk\pic(X)$, and $c$ is defined in \eqref{eq:leading constant} with $k=b-1$,  $C_{M,\mathbf d}$ given by \eqref{eq:Campana C_M},  and
$\varpi_1,\dots,\varpi_s$ given by \eqref{eq:Campana varpi}. Moreover, the restriction $\pic(X)\to\pic(V)$ is an isomorphism as $n_i\geq 3$ for $1\leq i\leq t$.
\end{proof}

\bibliographystyle{alpha}
\bibliography{CampanaPointsProjective}
\end{document}